\documentclass[11pt]{amsart}
\usepackage{amssymb,amsmath,graphics,verbatim}
\usepackage{latexsym}
\usepackage{eucal}
\usepackage{a4wide}

\usepackage{color}

\newtheorem{theorem}{Theorem}
\newtheorem{lemma}[theorem]{Lemma}
\newtheorem{prop}[theorem]{Proposition}
\newtheorem{cor}[theorem]{Corollary}
\theoremstyle{definition}
\newtheorem{definition}[theorem]{Definition}
\theoremstyle{remark}
\newtheorem{remark}[theorem]{Remark}

\newcommand{\mc}{\mathcal}
\newcommand{\rr}{\mathbb{R}}
\newcommand{\nn}{\mathbb{N}}
\newcommand{\cc}{\mathbb{C}}
\newcommand{\hh}{\mathbb{H}}
\newcommand{\zz}{\mathbb{Z}}

\newcommand{\eps}{\epsilon}

\newcommand{\pl}{\partial}
\newcommand{\x}{\times}

\newcommand{\til}{\widetilde}
\newcommand{\bbar}{\overline}

\newcommand{\cjd}{\rangle}
\newcommand{\cjg}{\langle}

\newcommand{\demi}{\tfrac{1}{2}}

\newcommand{\grad}{\textrm{grad}}

\newcommand{\chs}{{\rm cs}}

\newcommand{\ad}{\mathrm{ad}}

\newcommand{\cC}{\mathcal{C}}
\newcommand{\cD}{\mc{D}}
\newcommand{\cF}{\mathcal{F}}

\newcommand{\cL}{\mc{L}}
\newcommand{\cLg}{\cL}
\newcommand{\cM}{\mc{M}}
\newcommand{\cMU}{\cM^\cU}
\newcommand{\cU}{\mathcal{U}}

\newcommand{\cX}{\mathcal{X}}

\newcommand{\cun}{\cC^\infty}
\newcommand{\dbar}{\bar{\partial}}
\newcommand{\Diff}{\mathcal{D}}
\newcommand{\dtheta}{\dot{\theta}}
\newcommand{\End}{\mathrm{End}}
\newcommand{\hg}{\hat{g}}
\newcommand{\hS}{\hat{S}}

\newcommand{\lan}{\langle}
\newcommand{\Mm}{\cM_{-1}}
\newcommand{\na}{\nabla}

\newcommand{\oX}{\overline{X}}
\newcommand{\Psldc}{{\mathrm{PSL}_2(\cc)}} 
\newcommand{\pt}{{\partial_t}}

\newcommand{\Tr}{\mathrm{Tr}}
\newcommand{\tr}{\mathrm{tr}}

\newcommand{\tY}{\tilde{Y}}
\newcommand{\tZ}{\tilde{Z}}
\newcommand{\cs}{{\rm cs}}
\newcommand{\CS}{{\rm CS}}
\newcommand{\CSP}{{\rm CS}^\Psldc}

\newcommand{\bg}{\mathbf{g}}
\newcommand{\dH}{\dot{H}}
\newcommand{\Mdc}{\mathrm{M}_2(\cc)}
\newcommand{\Mtc}{\mathrm{M}_3(\cc)}
\newcommand{\ps}{{\partial_s}}
\newcommand{\px}{\partial_{x}}
\newcommand{\ran}{\rangle}

\newcommand{\scal}{\mathrm{scal}}
\newcommand{\sldc}{\mathrm{sl}_2(\cc)}
\newcommand{\sot}{\mathrm{so(3)}}
\newcommand{\Sot}{\mathrm{SO(3)}}
\newcommand{\Tei}{\mathcal{T}}
\newcommand{\Teig}{\Tei}

\newcommand{\vol}{\mathrm{dvol}}
\newcommand{\Vol}{\mathrm{Vol}}
\newcommand{\rot}{\mathrm{curl}}

\newcommand{\ext}{{\mathrm{ext}}}
\newcommand{\homega}{\hat{\omega}}
\newcommand{\WP}{\mathrm{WP}}

\begin{document}
\title{Chern-Simons line bundle on Teichm\"uller space}
\date{\today}
\author{Colin Guillarmou}
\address{DMA, U.M.R.\ 8553 CNRS\\
Ecole Normale Sup\'erieure\\
45 rue d'Ulm\\ 
F 75230 Paris cedex 05 \\France}
\email{cguillar@dma.ens.fr}
\author{Sergiu Moroianu}
\address{Institutul de Matematic\u{a} al Academiei Rom\^{a}ne\\
P.O.\ Box 1-764\\
RO-014700 Bucharest, Romania}
\email{moroianu@alum.mit.edu}


\begin{abstract}
Let $X$ be a non-compact geometrically finite hyperbolic $3$-manifold without cusps of rank $1$. 
The deformation space $\mc{H}$ of $X$ can  
be identified with the Teichm\"uller space $\mc{T}$ of the conformal boundary of $X$ as the graph 
of a section in $T^*\mc{T}$. 
We construct  a Hermitian  holomorphic line bundle $\mc{L}$ on $\mc{T}$,  
with curvature equal to a multiple of the Weil-Petersson symplectic form.  
This bundle has a canonical holomorphic section defined by $e^{\frac{1}{\pi}{\rm Vol}_R(X)+2\pi i\CS(X)}$ 
where ${\rm Vol}_R(X)$ is the renormalized volume of $X$ and $\CS(X)$ is the Chern-Simons invariant 
of $X$. This section is parallel on $\mc{H}$  for the Hermitian connection modified by the $(1,0)$ component of the 
Liouville form on $T^*\mc{T}$. As applications, we deduce 
that $\mc{H}$ is Lagrangian in $T^*\mc{T}$, and  that ${\rm Vol}_R(X)$ is a K\"ahler potential 
for the Weil-Petersson metric on $\mc{T}$ and on its quotient by a certain subgroup of the mapping class group.
For the Schottky uniformisation, we use a formula of Zograf to construct an explicit isomorphism of holomorphic Hermitian line bundles between  $\mc{L}^{-1}$ and the sixth power of the determinant line bundle. 
\end{abstract}
\maketitle

\section{Introduction}
In \cite{ChSi}, S.S. Chern and J. Simons defined secondary characteristic classes of connections on principal bundles, arising from Chern-Weil theory. Their work has been extensively developed to 
what is now called Chern-Simons theory, with many applications in geometry and topology, but also in theoretical physics. 
For a Riemannian oriented $3$-manifold $X$, the Chern-Simons invariant $\CS(\omega,S)$ of 
the Levi-Civita connection form $\omega$ in an orthonormal frame $S$ is given by the integral of the $3$-form on $X$ 
 \[\tfrac{1}{16\pi^2} \Tr (\omega\wedge d\omega+\tfrac{2}{3}\omega\wedge \omega\wedge \omega).\]
 On closed $3$-manifolds, the invariant ${\rm CS}(\omega)$ is independent of $S$ up to integers.
  By the Atiyah-Patodi-Singer theorem 
for the signature operator, the Chern-Simons invariant of the Levi-Civita connection is related to the eta invariant by the identity $3\eta\equiv 2{\rm CS}$  modulo  $\zz$ (see for instance \cite{Yo}).

The theory has been extended to ${\rm SU}(2)$ flat connections on compact $3$-manifolds with boundary by Ramadas-Singer-Weitsman \cite{RSW}, in which case $\rm{CS}(\omega)$ does depend on  the boundary value of the section 
$S$. The Chern-Simons invariant $e^{2\pi i{\rm CS}(\cdot)}$ can be viewed as a section of a complex line bundle (with a Hermitian structure) over the moduli space of flat ${\rm SU}(2)$ connections on the boundary surface. 
They proved that this bundle is isomorphic to the determinant line bundle introduced by Quillen \cite{Qu}. 
Some more systematic studies and extensions of the Chern-Simons bundle have  been  developed by Freed \cite{Freed} and Kirk-Klassen \cite{KiKl}. One contribution of our present work is to give an explicit isomorphism 
between these Hermitian holomorphic line bundles in the Schottky setting. 

An interesting field of applications of Chern-Simons theory is for hyperbolic $3$-manifolds $X=\Gamma\backslash\hh^3$, which possess a natural flat connection $\theta$ over a principal ${\rm PSL}_2(\cc)$-bundle. For closed manifolds, Yoshida \cite{Yo}Ê defined the $\Psldc$-Chern-Simons invariant as above by 
\[ {\rm CS}(\theta)=-\tfrac{1}{16\pi^2}\int_X S^{*}\left(\Tr (\theta\wedge d\theta+\tfrac{2}{3}\theta\wedge \theta\wedge \theta)\right)\]
where $S: X\to P$ are particular sections coming from the frame bundle over $X$.
This is a complex number with imaginary part $-\frac{1}{2\pi^2}\Vol(X)$, and real part 
equal to the Chern-Simons invariant of the Levi-Civita connection on the frame bundle. 
Up to the contribution of a link in $X$, the function $F:=\exp(\frac{2}{\pi}\Vol(M)+4\pi i {\rm CS}(M))$ 
extends to a holomorphic function on a natural deformation space containing closed hyperbolic manifolds as a discrete set.

Our setting in this paper is that of $3$-dimensional geometrically finite hyperbolic manifolds $X$ without rank $1$ cusps,  in particular \emph{convex co-compact hyperbolic} manifolds, which are conformally compactifiable to a smooth manifold with boundary. Typical examples are quotients of $\hh^3$ by quasi-Fuchsian or Schottky groups.
The ends of $X$ are either funnels or rank $2$ cusps. The funnels have a conformal 
boundary, which is a disjoint union of compact Riemann surfaces forming the \emph{conformal boundary} $M$ of $X$. 
The deformation space of $X$  is essentially the deformation space of its conformal boundary, i.e.\ Teichm\"uller space. Before defining a Chern-Simons invariant, it is natural to ask about a replacement of the volume in this case. For Einstein conformally compact manifolds, the notion of \emph{renormalized volume} $\Vol_R(X)$ has been introduced by Henningson-Skenderis \cite{HeSk} in the physics literature and by Graham \cite{Gr} in the mathematical
literature.  In the particular setting of hyperbolic $3$-manifolds, this has been studied by Krasnov \cite{Kra}
and extended by Takhtajan-Teo \cite{TakTeo}, in relation with earlier work of Takhtajan-Zograf \cite{TakZo}, 
to show that $\Vol_R$ is a K\"ahler potential for the Weil-Petersson metric in Schottky and quasi-Fuchsian settings. Krasnov and Schlenker \cite{KrSc} gave a more geometric proof of this, using the Schl\"afli formula on convex co-compact hyperbolic $3$-manifolds to compute the variation of $\Vol_R$ in the deformation space (notice that there were works of Anderson \cite{An} and Albin \cite{Alb} on the variation of $\Vol_R$ in even dimensional Einstein settings).

Before we introduce the Chern-Simons invariant in our setting, let us first recall the definition of $\Vol_R$ used by Krasnov-Schlenker \cite{KrSc}.
A \emph{hyperbolic funnel} is some collar $(0,\eps)_x\x M$ equipped with a metric 
\begin{align}\label{funneldef}
g=\frac{dx^2+h(x)}{x^2},&& h(x) \in C^\infty(M,S^2_+T^*M),&&
h(x)=h_0\left(({\rm Id}+\tfrac{x^2}{2}A)\cdot ,({\rm Id}+\tfrac{x^2}{2}A)\cdot \right)
\end{align}
where $M$ is a Riemann surface of genus $\geq 2$ with a hyperbolic metric $h_0$,  $A$ is an endomorphism of $TM$ satisfying ${\rm div}_{h_0}A=0$, and $\Tr(A)=-\demi \scal_{h_0}$. 
The metric $g$ on the funnel is of constant sectional curvature $-1$, and every end of $X$ which is not a cusp end 
is isometric to such a hyperbolic funnel, see \cite{FefGr,KrSc}. 
A couple $(h_0,A_{0})$ can be considered as an element of $T^*_{h_0}\mc{T}$, 
if $A_{0}=A-\demi \tr(A){\rm Id}$ is the trace-free part of the divergence-free tensor $A$.
We therefore identify the tangent bundle $T^*\mc{T}$ of $\Tei$
with the set of hyperbolic funnels modulo the action of the group $\mc{D}_0(M)$, 
acting trivially in the $x$ variable.
Let $x$ be a smooth positive function  on $X$ which extends the function $x$ defined in each funnel by \eqref{funneldef}, and is equal to $1$ in each cusp end. The renormalized volume of $(X,g)$ is defined by
\[\Vol_{R}(X):={\rm FP}_{\eps\to 0}\int_{x>\eps}  \vol_g\]
where ${\rm FP}$ means finite-part (i.e.\ the coefficient of $\eps^0$ in the asymptotic expansion as $\eps\to 0$).

If $\omega$ is the $\sot$-valued Levi-Civita connection $1$-form on $X$ in an oriented orthonormal frame $S=(S_1,S_2,S_3)$, we define 
\begin{equation}\label{defiCS} 
\CS(g,S):= -\tfrac{1}{16\pi^2}{\rm FP}_{\eps\to 0} \int_{x>\eps} \Tr(\omega\wedge d\omega+\tfrac{2}{3}\omega\wedge \omega\wedge\omega). 
\end{equation}
We ask that $S$ be even to the first order at $\{x=0\}$ and also that, in each cusp end, $S$ be parallel 
in the direction of the vector field pointing towards to cusp point. 
Equipped with the conformal metric $\hat{g}:=x^2g$, the manifold $X$ extends to a smooth Riemannian manifold 
$\bbar{X}=X\cup M$ with boundary $M$. The Chern-Simons invariant $\CS(\hat{g},\hat{S})$ is therefore 
well defined if $\hat{S}=x^{-1}S$ is an orthonormal frame for $\hat{g}$. 
We  define the $\Psldc$ Chern-Simons invariant $\CS^{\Psldc}(g,S)$ on $(X,g)$ by the renormalized integral \eqref{defiCS} where we replace $\omega$ by 
the complex-valued connection form $\theta:=\omega+iT$; here $T$ is the $\sot$-valued $1$-form defined by 
$T_{ij}(V):=g(V\x S_j,S_i)$ and $\x$ is the vector product with respect to the metric $g$. 
There exists a natural flat connection
on a $\Psldc$ principal bundle $F^\cc(X)$ over $X$ (which can be seen as a complexified frame bundle), 
with $\sldc$-valued connection $1$-form $\Theta$, 
and we show that $\CS^{\Psldc}(g,S)$ also equals the renormalized integral of the 
pull-back of the Chern-Simons form $-\frac{1}{4\pi^2}\Tr(\Theta\wedge d\Theta+
\tfrac{2}{3}\Theta^3)$ of the flat connection $\Theta$, see Section \ref{Killingfields}. 
We first show 
\begin{prop}\label{theor1}
On a geometrically finite  hyperbolic $3$-manifold $(X,g)$ without rank $1$ cusps, one has 
$\CS(g,S)=\CS(\hat{g},\hat{S})$, and  
\begin{equation}\label{CSPgS}
\CSP(g,S)=-\tfrac{i}{2\pi^2}{\rm Vol}_R(X)+\tfrac{i}{4\pi}\chi(M)+\CS(g,S)
\end{equation}
where $\chi(M)$ is the Euler characteristic of the conformal boundary $M$.
\end{prop}
The relation between $\CS(g,S)$ and $\CS(\hat{g},\hat{S})$ comes rather easily from the conformal change formula 
in the Chern-Simons form  (the boundary term turns out to not contribute), while \eqref{CSPgS} is a generalization 
of a formula in Yoshida \cite{Yo}, but we give an independent easy proof. Similar identities to \eqref{CSPgS} can be found in the physics literature (see for instance \cite{Kra2}).
 
Like the function $F$ of Yoshida, it is natural to consider the variation of $\CSP(g,S)$ in the 
set of convex co-compact hyperbolic $3$-manifolds, especially since, in contrast with the finite volume case, 
there is a finite dimensional deformation space of smooth hyperbolic $3$-manifolds, which  essentially coincides
with the Teichm\"uller space of their conformal boundaries. One of the problems, related to the work of Ramadas-Singer-Weitsman \cite{RSW} is that $e^{2\pi i\CSP(g,S)}$ depends on the choice of the frame $S$, since $X$ is not closed. This leads us to define a complex line bundle $\mc{L}$ over Teichm\"uller space  $\mc{T}$ 
of Riemann surfaces of a fixed genus, in which $e^{2\pi i\CSP}$ and $e^{2\pi i\CS}$ are sections. 

Let $\Tei$ be the Teichm\"uller space of a (not necessarily connected) oriented Riemann surface $M$ of genus 
$\bg=(g_1,\dots,g_N)$, $g_j\geq 2$, defined as the space of hyperbolic metrics on $M$ modulo the group $\mc{D}_0(M)$
of diffeomorphisms isotopic to the identity. This is a complex simply connected manifold of complex 
dimension $3|\bg|-3$, equipped with a natural K\"ahler metric called the Weil-Petersson metric (see Subsection \ref{teichmuller}). The \emph{mapping class group} ${\rm Mod}$ of isotopy classes of orientation preserving diffeomorphisms of $M$ acts properly discontinuously on $\Tei$.  
Let $(X,g)$ be a geometrically finite hyperbolic $3$-manifold without cusp of rank $1$, 
with conformal boundary $M$. By Theorem 3.1 of \cite{Marden}, there is a smooth map
$\Phi$ from $\Tei$ to the set of geometrically finite hyperbolic metrics on $X$ (up to diffeomorphisms of $X$) such that
 the conformal boundary of $\Phi(h)$ is $(M,h)$ for any $h\in\Tei$. 
The subgroup ${\rm Mod}_X$ of ${\rm Mod}$ consisting of elements which extend to diffeomorphisms
on $\oX$ homotopic to the identity acts freely, properly discontinuously on $\Tei$ and the quotient is
a complex manifold of dimension $3|\bg|-3$. The map $\Phi$ is invariant under the action of 
${\rm Mod}_X$ and the deformation space $\Tei_X$ of $X$ is identified with a 
quotient of the Teichm\"uller space $\Tei_X=\Tei/{\rm Mod}_X$, see \cite[Th. 3.1]{Marden} .

\begin{theorem}\label{theo1b}
Let $(X,g)$ be a geometrically finite hyperbolic $3$-manifold without rank $1$ cusp, and with conformal
boundary $M$. There exists a holomorphic Hermitian line bundle $\mc{L}$ over $\Tei$ equipped with a 
Hermitian connection $\nabla^{\mc{L}}$, with curvature given  by $\frac{i}{8\pi}$ times the Weil-Petersson symplectic form 
$\omega_{\rm WP}$ on $\Tei$. The bundle $\mc{L}$ with its connection descend to $\Tei_X$ and if $g_h=\Phi(h)$ is the 
geometrically finite hyperbolic metric with conformal boundary $h\in\Tei$, then 
$h\to e^{2\pi i\CS(g_h,\cdot)}$ is a global section of $\mc{L}$.   
\end{theorem}
The line bundle is defined using the cocycle which appears in the Chern-Simons action under gauge transformations, this is explained in Subsection \ref{linebundle}. 
We remark that the computation of the curvature of $\mc{L}$ reduces to the computation of the curvature of the vertical tangent bundle in a fibration related to the universal Teichm\"uller curve over $\mc{T}$, and we show that the fiberwise integral of the first Pontrjagin form of this bundle is given by the Weil-Petersson form, which is similar to a result of Wolpert \cite{Wol}.  An analogous line bundle, but in a more general setting, has been recently studied by Bunke \cite{Bunke}.

Since funnels can be identified to elements in $T^*\Tei$, the map $\Phi$ described above induces a section $\sigma$ of the bundle $T^*\Tei$ (which descends to $T^*\Tei_X$) by assigning to $h\in \Tei$ the funnels of $\Phi(h)$. The image of $\sigma$
\[\mc{H}:=\{\sigma(h)\in T^*\Tei_X,  h \in \Tei_X\}\]
identifies the set of geometrically finite hyperbolic metrics on $X$ as a graph in $T^*\Tei_X$. 
 
Let us still denote by $\mc{L}$ the Chern-Simons line bundle pulled-back to $T^*\Tei$ 
by the projection $\pi_{\mc{T}}:T^*\Tei\to \Tei$, 
and define a modified connection 
\begin{equation}\label{newconnection}
\nabla^{\mu}:= \nabla^{\mc{L}}+\tfrac{2}{\pi}\mu^{1,0}
\end{equation}
on $\mc{L}$ over $T^*\Tei$, where $\mu^{1,0}$ is the $(1,0)$ part of the Liouville $1$-form $\mu$ 
on $T^*\Tei$. As before, the connection descends to $T^*\mc{T}_X$, and notice that it is not 
Hermitian (since $\mu^{1,0}$ is not purely imaginary) but $\nabla^\mu$ and $\nabla^{\mc{L}}$ induce the same holomorphic structure on $\mc{L}$. 

By Theorem \ref{theo1b} and Proposition \ref{theor1},   $e^{2\pi i\CS^{\Psldc}}$ is a section of $\mc{L}$ on 
$\Tei_X$, its pull-back by $\pi_{\Tei}$ also gives a section of $\mc{L}$ on $\mc{H}$, which we 
still denote $e^{2\pi i\CS^{\Psldc}}$.

\begin{theorem}\label{theo3}
Let $V\in T\mc{H}$ be a vector field tangent to $\mc{H}$, then $\nabla^\mu_V e^{2\pi i\CS^{\Psldc}}=0$, i.e. $\nabla^\mu$ is flat on $\mc{H}\subset T^*\Tei_X$. 
\end{theorem}

The curvature of $\nabla^\mu$ vanishes on $\mc{H}$ by Theorem \ref{theo3} while
the curvature of $\nabla^{\mc{L}}$ is $\tfrac{i}{8\pi}\omega_{\rm WP}$ (by Theorem \ref{theo1b}). 
By considering the real and imaginary parts of these curvature identities, 
we obtain as a direct corollary :
\begin{cor}\label{firstcor}
The manifold $\mc{H}$ is Lagrangian in $T^*\Tei_X$ for the Liouville symplectic form $\mu$ and 
$d({\rm Vol}_R\circ \sigma)=-\frac{1}{4}\mu$ on $\mc{H}$. The renormalized volume is a K\"ahler potential
for Weil-Petersson metric on $\Tei_X$: 
\[\bar{\pl}\pl ({\rm Vol}_R\circ \sigma) = \tfrac{i}{16}\omega_{\rm WP}.\]
\end{cor}

Our final result relates the Chern-Simons line bundle $\mc{L}$ to the Quillen determinant line bundle $\det \pl$  
of $\pl$ on functions in the particular case of Schottky hyperbolic manifolds. If $M$ is a connected surface of
genus $\bg\geq 2$, one can realize any complex structure on $M$ as a quotient of an open set $\Omega_\Gamma\subset \cc$ by a 
Schottky group $\Gamma\subset \Psldc$ and using a marking $\alpha_1,\dots,\alpha_\bg$ of $\pi_1(M)$ and a certain normalization, there is complex manifold 
$\mathfrak{S}$, called the Schottky space, of such groups. 
This is isomorphic to $\mc{T}_X$, where $X:=\Gamma\backslash \hh^3$ is the solid torus bounding $M$ in which the curves $\alpha_j$ are contractible. 
The Chern-Simons line bundle $\mc{L}$ can then be defined on $\mathfrak{S}$.
The Quillen determinant bundle $\det\pl$ is equipped with its Quillen metric  and a natural holomorphic structure induced by $\mathfrak{S}$ (see Subsection \eqref{detline}), therefore inducing 
a Hermitian connection compatible with the holomorphic structure. Moreover, there is a canonical section of
$\det \pl=\Lambda^{\bg}({\rm coker}\,\pl)$ given by $\varphi:=\varphi_1\wedge \dots\wedge\varphi_\bg$ where $\varphi_j$ are holomorphic $1$-forms on $M$
normalized by the marking through the requirement $\int_{\alpha_j}\varphi_k=\delta_{jk}$. 
Using a formula of Zograf \cite{Zo1,Zo2}, we show
\begin{theorem}
There is an explicit isometric isomorphism of holomorphic Hermitian line bundles between 
the inverse $\mc{L}^{-1}$ of the Chern-Simons line bundle and the $6$-th power $(\det\pl)^{\otimes 6}$ of the determinant line bundle $\det\pl $, given by   
\[ 
 (F \varphi)^{\otimes 6}  \mapsto e^{-2\pi i\CSP}. 
\]  
Here $\varphi$ is the canonical section of $\det\pl$ defined above, $c_\bg$ is a constant,
and $F$ is a holomorphic function on $\mathfrak{S}$ which is given, on the open set where the product converges absolutely, by
\[F(\Gamma)= c_\bg \prod_{\{\gamma\}}\prod^\infty_{m=0} (1-
q_\gamma^{1+m}),\]
where $q_\gamma$ is the multiplier of $\gamma\in\Gamma$, $\{\gamma\}$ runs over all distinct primitive conjugacy classes in $\Gamma\in \mathfrak{S}$ except the identity.
\end{theorem} 

\subsection*{Novelties and perspectives}
Our main contribution in this work is to introduce the Chern-Simons theory and its line bundle over 
Teichm\"uller space in relation with Kleinian groups. The strength of this construction appears through a variety of applications to Teichm\"uller theory in essentially the most general setting, all at once and self-contained. For example, the  property of the renormalized volume of being a K\"ahler potential for the Weil-Petersson metric, previously 
known in the particular cases of Schottky and quasi-Fuchsian groups \cite{Kra,TakTeo,TakZo,KrSc}, follows directly from our Chern-Simons approach for all geometrically finite Kleinian groups without cusps of rank $1$ (for instance, the proof in \cite{KrSc} is based on an explicit computation at the Fuchsian locus and does not seem to be extendable to general groups).
In fact, finding K\"ahler potentials for the Weil-Petersson metric starting from a general Kleinian cobordism means  more than just a generalisation of the quasi-Fuchsian and Schottky cases.
Indeed, the Chern-Simons bundle $\mc{L}$ is a ``prequantum bundle'' and together with the canonical holomorphic sections $e^{2\pi i\CS^{\Psldc}}$ corresponding to each hyperbolic cobordism,  it opens a vista towards a geometric definition of a Topological Quantum Field Theory through the quantization of  Teichm\"uller space. We shall pursue this question elsewhere.

The existence of a non-explicit isomorphism between the Chern-Simons bundle on the (compact) moduli space of ${\rm SU}(2)$ flat connections and the determinant line bundle  
was discovered in \cite{RSW}. In contrast, in our non-compact ${\rm PSL}_2(\cc)$ setting we find an \emph{explicit} isomorphism, involving a remarkable formula of Zograf on Schottky space, which as far as we know is the first of its kind; 
we intend to generalize this to all convex co-compact groups.     

More generally, we expect the results of this paper to extend to all geometrically finite hyperbolic $3$-manifolds. Several technical difficulties appear when we perform our analysis to cusps of rank $1$, this will be carried out elsewhere.

\noindent\textbf{Organization of the paper}: In order to simplify as much as possible the presentation,
we discuss the case of convex co-compact manifolds in the main body of the paper, and
add an appendix including the case of rank $2$ cusps. In several parts of the paper, we also consider the more general setting of asymptotically hyperbolic manifolds, which only have asymptotic constant curvature near
$\infty$. 
The paper splits in two parts: in Section 2--6, we introduce Chern-Simons invariants associated to the Levi-Civita connection and to a certain complexification thereof on asymptotically hyperbolic $3$-manifolds with totally geodesic boundary,     
and we study their relationship with the renormalized volume in the case of convex co-compact hyperbolic metrics.
In the second part, Section 7--9, we define the Chern-Simons line bundle over $\mc{T}$, its connections, we compute the variation of our Chern-Simons invariants, and derive the implications on the Weil-Petersson metric and on the determinant line bundle.

\section{Asymptotically hyperbolic manifolds}

Let $(X,g)$ be an $(n+1)$-dimensional \emph{asymptotically hyperbolic} manifold, i.e., $X$ is the interior of a compact smooth manifold with boundary $\oX$, and there exists a smooth boundary-defining function $x$ such that near the boundary $\{x=0\}$ the Riemannian metric $g$ has the form
\[g=\frac{dx^2+h(x)}{x^2}\] 
in a product decomposition $[0,\epsilon)_x\times M\hookrightarrow \oX$ near the boundary $M=\pl\oX$, for some smooth one-parameter family $h(x)$ of metrics on $M$. A boundary-defining function $x$ inducing this product decomposition satisfies $|dx|_{x^2g}=1$ near $\pl \oX$, and is called a \emph{geodesic boundary defining function}.
When $\pl_xh(x)_{|x=0}=0$, the boundary $M$ is totally geodesic for the metric $\hat{g}:=x^2g$ and  we shall say
that $g$ has \emph{totally geodesic boundary}. This condition is shown in \cite{GuiDMJ} to be invariant with respect to the choice of $x$. Examples of asymptotically hyperbolic manifold with totally geodesic boundary are the hyperbolic space $\hh^{n+1}$, or more generally convex co-compact hyperbolic manifolds (cf.\ Eq. \eqref{metrhyp}).
The \emph{conformal boundary} of $(X,g)$ is the compact manifold $M=\pl\oX$ equipped with the conformal class $\{h_0\}$ of $h_0:=h(0)=x^2g|_{TM}$.

\subsection{Convex co-compact hyperbolic quotients}
Let $X$ be an oriented complete hyperbolic $3$-manifold, equipped with its
constant curvature metric $g$. The universal cover $\tilde{X}$ is isometric to the $3$-dimensional hyperbolic space $\hh^3$, and the deck transformation group is conjugated via this isometry to a Kleinian group $\Gamma\subset\Psldc$ (we recall below that $\Psldc$ can be viewed as the group of orientation-preserving isometries of $\hh^3$). In this way we get a representation of the fundamental group
\begin{equation}\label{defrho}
\rho: \pi_1(X)\to {\rm PSL}_2(\cc)
\end{equation}
with image $\Gamma$, well-defined up to conjugation.

We say that $X$ is \emph{convex co-compact hyperbolic} if it is isometric to 
$\Gamma\backslash\hh^3$ for some discrete group $\Gamma\subset {\rm PSL}_2(\cc)$
with no elliptic, nor parabolic transformations, such that $\Gamma$ admits a fundamental domain in $\hh^3$ with a finite number of sides.   
Then the manifold $X$ has a smooth compactification into a manifold $\bbar{X}$, with boundary $M$ which is a disjoint union of compact Riemann 
surfaces. The boundary can be realized as the quotient $\Gamma \backslash \Omega(\Gamma)$ 
where $\Omega(\Gamma)\subset S^2$ is the domain of discontinuity of the convex co-compact subgroup
$\Gamma$, acting as conformal transformations on the sphere $S^2$. 
Each connected component of $M$ has a projective structure induced by the group $\Gamma\subset {\rm PSL_2(\cc)}$. 
It is proved in \cite{FefGr,KrSc} that the constant sectional curvature condition implies the following structure for the metric near infinity: there exists a product decomposition $[0,\eps)_x\x M$ of $\bbar{X}$ near $M$, induced by the choice of a geodesic boundary-defining function $x$ of $M$, a metric $h_0$ on $M$ and a symmetric endomorphism $A$ of $TM$ such that the metric $g$ is of the form
\begin{align}\label{metrhyp}
g=\frac{dx^2+h(x)}{x^2}, && h(x)=h_0\left((1+\tfrac{x^2}{2}A)\cdot\, ,(1+\tfrac{x^2}{2}A)\cdot\right),
\end{align}
and moreover $A$ satisfies 
\begin{align}\label{condA}
\Tr(A)=-\demi\scal_{h_0},&& {d^\nabla}^*A=0.
\end{align}

\subsection{Tangent, cotangent and frame bundles} \label{para22}
There exists a smooth vector bundle over $\oX$ spanned over $\cun(\oX)$ by smooth vector 
fields vanishing on the boundary $\pl\oX$ (those are locally spanned near $\pl\oX$ 
by $x\pl_x$ and $x\pl_y$ if $x$ is a boundary defining function and $y$ are coordinates on the boundary), we denote by ${^0T}\oX$ this bundle.
Its dual is denoted ${^0T}^*\oX$ and is locally spanned over $\cun(\oX)$ by the forms $dx/x$ and $dy/x$. 
An asymptotically hyperbolic metric can be also defined to be a smooth section of the bundle of positive definite symmetric tensors 
$S_+^2({^0T}^*\oX)$ such that $|dx/x|_g=1$ at $\pl\oX$. 
The frame bundle $F_0(X)$ for an asymptotically hyperbolic metric $g$ is a $\Sot$-principal bundle and its sections are smooth $g$-orthogonal 
vector fields in ${^0T}\oX$. It is clearly canonically isomorphic to the frame bundle $F(X)$ of the compactified metric 
$\hat{g}:=x^2g$ if $x$ is a boundary defining function. A smooth frame $S\in F_0(X)$ is said to be \emph{even to first order} if, in local coordinates $(y_1,y_2,x)$ near $\pl \oX$ induced by any geodesic defining function $x$, $S$ is of the form $S=x(u_1\pl_{y_1}+u_2\pl_{y_2}+u_3\pl_x)$ where $u_j$ are such that $\pl_xu_j|_{M}=0$, or equivalently
$[\pl_x,\hat{S}]|_{M}=0$ if $\hat{S}:=x^{-1}S$ is the related frame for $\hat{g}:=x^2g$.  
In general, we refer the reader to \cite{MaMel,Ma} for more details about the $0$-structures and bundles. 

\subsection{Orientation convention}

For an oriented asymptotically hyperbolic manifold the orientation of the boundary at infinity $M$ is defined by the requirement that $(\px,Y_1,Y_2)$ is a positive frame on $X$ if and only if $(Y_1,Y_2)$ is a positive frame on $M$.
With this convention, Stokes's formula gives
\[\int_{\oX} d\alpha=-\int_M\alpha\]
for every $\alpha\in\cun(\oX,\Lambda^2)$.

\subsection{Renormalized integrals}
Let  $\omega\in x^{-N}C^\infty(\bbar{X}, \Lambda^n\bbar{X})+C^\infty(\bbar{X},\Lambda^n\bbar{X})$ for some $N\in \rr^+$.
 The 0-integral (or renormalized integral) of $\omega$ on $X$ is defined by 
\[\int_X^0 w:= {\rm FP}_{\eps\to 0}\int_{x>\eps}\omega \]
where ${\rm FP}$ denotes the finite part, i.e., the coefficient of $\eps^0$ in the expansion of the integral at $\eps=0$.
This is independent of the choice of function $x$ when $N$ is not integer or $N>-1$ but it depends a priori on the choice of $x$ when $N$ is a negative integer. In the present paper, we shall always fix the geodesic boundary defining function $x$ so that the induced metric $h_0=x^2g|_{TM}$ is the unique hyperbolic metric in its conformal class.
More generally, one can define renormalized integrals of polyhomogeneous forms but this will not be used here. We refer the reader  \cite{Alb,GuAJM} for detailed discussions on this topic. An example which has been introduced by Henningson-Skenderis \cite{HeSk} and 
Graham \cite{Gr} for asymptotically hyperbolic Einstein manifolds is the renormalized volume defined by
\[{\rm Vol}_R(X):=\int^0_X \vol_g\]
where $\vol_g$ is the volume form on $(X,g)$. 

\section{The bundle of infinitesimal Killing vector fields for hyperbolic manifolds}\label{Killingfields}

The hyperbolic $3$-space $\hh^3$ can be  viewed as a subset of quaternions 
\begin{align*}
\hh^3\simeq \{y_1+iy_2+y_3j; y_3>0, y_1,y_2\in \rr\},&&g_{\hh^3}=\frac{dy^2}{y_3^2}.
\end{align*}
The action of $\gamma=\begin{bmatrix} a & b\\ c & d
\end{bmatrix}\in {\rm PSL}_2(\cc)$ on $\zeta=y_1+iy_2+y_3j\in \hh^3$ is given by 
\[\gamma.\zeta=(a\zeta+b)(c\zeta+d)^{-1}.\]
This action identifies ${\rm PSL}_2(\cc)$ with the group of oriented isometries of $\hh^3$, which is diffeomorphic to the frame bundle $F(\hh^3)=\hh^3\x {\rm SO}(3)$ of $\hh^3$ via the map
\begin{align*} 
\Phi: {\rm PSL}_{2}(\cc) \to F(\hh^3),&& \gamma \mapsto (\gamma.j,\gamma_*(\pl_{y_1},\pl_{y_2},\pl_{y_3})).
\end{align*}
There exists a natural embedding 
\begin{align}\label{tildeq}
\til{q}: F(\hh^3)\to \hh^3\x {\rm PSL}_2(\cc),&& (m,V_m)\mapsto (m, \Phi^{-1}(m,V_m)).
\end{align}
which is equivariant with respect to the right action of $\Sot$. If $X=\Gamma\backslash \hh^3$ is an oriented hyperbolic quotient, $\til{q}$ descends to a bundle map
\begin{equation}\label{mapq}
q:F(X)\to \hh^3\x_\Gamma {\rm PSL}_2(\cc)=:F^\cc(X).
\end{equation}
where $F^\cc(X)$ is a principal bundle over $X$ with fiber $\Psldc$. The trivial flat connection 
on the product $\hh^3\x \Psldc$ also descends to a flat connection on $F^\cc(X)$, 
denoted $\theta$ (i.e., a $\sldc$-valued $1$-form on $F^\cc(X)$), 
with holonomy representation conjugated to $\rho$, where $\rho$ is defined in \eqref{defrho}. 

Let $\nabla$ be the Levi-Civita connection on $TX$ with respect to the hyperbolic metric $g$, and let \begin{align}\label{deft}
T\in\Lambda^1(X,\End(TX)), &&T_VW:=-V\x W,
\end{align}
where $\x$ is the vector product with respect to the metric $g$.
\begin{prop} \label{vbE}
The vector bundle $E(X)$ associated to the principal bundle $F^\cc(X)$ with respect to the adjoint representation 
is isomorphic, as a complex bundle, to the complexified tangent bundle $T_\cc X$. The connection induced 
by $\theta$ is $D:=\nabla+iT$.
\end{prop}
\begin{proof} 
The associated bundle with respect to the adjoint representation is given by 
\[E(X)=\hh^3\x_\Gamma \Psldc\times_{\Psldc} {\rm sl}(2,\cc)=\hh^3\x_\Gamma {\rm sl}(2,\cc)=(\hh^3\x \sldc)/\sim\]
where the equivalence relation is $[m,h]\sim[\gamma m,\gamma h\gamma^{-1}]$
for all $\gamma\in \Gamma$. We also have $TX=\Gamma\backslash T\hh^3$ where the action of $\Psldc$ on 
$T\hh^3$ is given by $\gamma.(m,v_m):=(\gamma m,{\gamma_*}(v_m))$. For every vector field $u$ on $X$ define its \emph{canonical lift} $s_u$
to $T_\cc X$ by 
\[s_u:=u+\tfrac{i}{2}\rot(u)\] 
where $\rot (u)=(*d u^\flat )^\sharp$ (the map $u\mapsto s_u$ is a first-order differential operator). Note that the sign in front of $\rot$ is different from the one used in \cite{HK}. 
For every $h\in\sldc$ let $\kappa_h$ be the Killing field on $\hh^3$ corresponding to the infinitesimal isometry $h$.
\begin{lemma}\label{lema2}
We have $\kappa_{ih}=-\tfrac{1}{2}\rot(\kappa_h)$, thus 
\begin{align*}
s_{\kappa_h}=\kappa_h-i\kappa_{ih},&& s_{\kappa_{ih}}=is_{\kappa_h}.
\end{align*}
\end{lemma}
\begin{proof}
Direct verification on a basis of $\sldc$, using the explicit formula for $\kappa_h$ at $q\in\hh^3$:
\[\kappa_h= b+ aq+qa-qcq\]
where $h=\begin{bmatrix} a&b\\ c&d\end{bmatrix}\in\sldc$ and $T_q \hh^3$ was identified with $\cc\oplus j\rr$.
\end{proof}

Define a vector bundle morphism
\begin{align*}\Psi:\cun(\hh^3,E(\hh^3))\to \cun(\hh^3,T_\cc\hh^3),&&(m,h)\mapsto s_{\kappa_h}(m)=\kappa_h(m)+\tfrac{i}{2}\rot(\kappa_h)(m).
\end{align*}
This map is injective, for a Killing field which vanishes at a point together with its curl must vanish identically. By dimensional reasons, $\Psi$ must be a bundle isomorphism. Moreover, $\Psi$ is $\Psldc$-equivariant in the sense that for all $\gamma\in \Psldc$ we have
$\Psi(\gamma m,\ad_\gamma h)=\gamma_*\Psi(m,h)$ (this is clear for the real part by definition, while for the imaginary part we use the fact that $\gamma$ is an isometry to commute it across $\rot$), hence $\Psi$ descends to the $\Gamma$-quotient as an isomorphism $\cun(X,E(X))\to \cun(X,T_\cc X)$. By Lemma \ref{lema2}, this isomorphism is compatible with the complex structures.
It remains to identify the push-forward $D$ of the flat connection from $E(X)$ to $T_\cc X$ under this map. 
It is enough to prove that on $\hh^3$ we have $D=\nabla+iT$, since both terms are $\Psldc$ invariant.

\begin{lemma}
Let $\kappa$ be a Killing vector field on an oriented $3$-manifold of constant sectional curvature $\epsilon$. Then the field $v:=-\tfrac12\rot(\kappa)$ is also Killing, and satisfies for every vector $U$
\begin{align*}
\nabla_U\kappa = U\times v,&&\nabla_U v= \epsilon U\times \kappa.
\end{align*}
\end{lemma}
\begin{proof}
Directly from the Koszul formula, the Levi-Civita covariant derivative of a Killing vector field $\kappa$ satisfies
$\lan \nabla_U\kappa, V\ran=\tfrac12 d\kappa^\sharp(U,V)$, which in dimension $3$ implies 
\[\nabla_U\kappa= -\tfrac12 U\times \rot(\kappa)=U\times v.\]
Let now $(U_1,U_2,U_3)$ be a radially parallel orthonormal frame near a point $p$, so that $\nabla_{U_i}U_j=0$ at $p$. On one hand, by assumption on the sectional curvatures, one has $\lan R_{U_1 U_2}\kappa,U_3\ran=0$ where $R$ is the curvature tensor of the metric. On the other hand, at the point $p$ we have
\begin{align*}
\lan R_{U_1 U_2}\kappa,U_3\ran=& U_1\lan \nabla_{U_2}\kappa,U_3\ran-U_2\lan \nabla_{U_1}\kappa,U_3\ran\\
=&U_1\lan v\times U_2,U_3\ran-U_2\lan v\times U_1,U_3\ran\\
=&\lan\nabla_{U_1}v, U_1\ran +\lan \nabla_{U_2} v,U_2\ran.
\end{align*}
Similarly, $\lan\nabla_{U_2}v, U_2\ran +\lan \nabla_{U_3} v,U_3\ran=0$ and $\lan\nabla_{U_3}v, U_3\ran +\lan \nabla_{U_1} v,U_1\ran=0$ so we deduce that $\lan\nabla_{U_j}v,U_j\ran=0$ at $p$. So $\nabla v$ is skew-symmetric at the (arbitrary) point $p$, or equivalently $v$ is Killing.

Since $\nabla_{U_i} \kappa=U_i\times v$ we see that 
\begin{align*}v=&\sum \lan v,U_i\ran U_i= \lan v,U_2\times U_3\ran U_1+\lan v,U_3\times U_1\ran U_2+\lan v,U_1\times U_2\ran U_3\\
=&\lan U_3\times v,U_2\ran U_1+\lan U_1\times v,U_3\ran U_2+\lan U_2\times v,U_1\ran U_3\\
=&\lan\nabla_{U_3} \kappa,U_2\ran U_1+\lan\nabla_{U_1} \kappa,U_3\ran U_2+\lan\nabla_{U_2} \kappa,U_1\ran U_3\\
\intertext{hence at the point $p$ where $U_j$ are parallel and commute, using $\lan\nabla_{U_2} \kappa,U_2\ran=0$,}
\lan \nabla_{U_2} v,U_1\ran=&\lan\nabla_{U_2}\nabla_{U_3} \kappa,U_2\ran=\lan R_{U_2 U_3}\kappa,U_2\ran = \epsilon \lan \kappa,U_3\ran=\epsilon\lan U_2\times \kappa,U_1\ran.
\end{align*}
Similarly, $\lan \nabla_{U_2} v,U_3\ran=\epsilon\lan U_2\times \kappa,U_3\ran$. Together with $\lan\nabla_{U_2} v,U_2\ran=0$ proved above, we deduce $\nabla_{U_2} v=U_2\times \kappa$. This identity clearly holds for any $U$ in place of $U_2$.
\end{proof}
For every $h\in\sldc$, the section $S_h:m\mapsto (m,h)$ is by definition a flat section in $E(\hh^3)$, so (again by definition) $D_U\Psi(S_h)=0$ for every vector $U$. Using the above lemma, we also have 
$(\nabla_U+iT(U))(\kappa_h+\tfrac{i}{2}\rot(\kappa_h))=0$. Thus the connections $\nabla+iT$ and $D$ have the same parallel (generating) sections, hence they coincide.
\end{proof}

\section{Chern-Simons forms and invariants} 

Let $Z$ be a manifold, $n\in\nn^*$ and $\theta\in\Lambda^1(Z,M_n(\cc))$ a matrix-valued $1$-form, and set $\Omega:=d\theta+\theta\wedge\theta$. Define 
\[\cs(\theta):=\Tr(\theta\wedge d\theta+\tfrac{2}{3}\theta^3)=\Tr(\theta\wedge \Omega-\tfrac{1}{3}\theta^3).\] 
(Notation: if $\alpha_j$ are $M_n(\cc)$-valued forms of degree $d_j$ on $Z$, $j=1,\ldots,k$, their exterior
product is defined by its action on vectors $V_1,\ldots,V_N$, $N:=\sum_{j=1}^k d_j$ as follows
\begin{align*}
\lefteqn{(\alpha_1\wedge\ldots\wedge \alpha_k)(V_1,\ldots,V_N)}&\\
&&:=\frac{1}{d_1!\ldots d_k!}
\sum_{\sigma\in\Sigma(N)} \epsilon(\sigma) \alpha_1(V_{\sigma(1)},\ldots, V_{\sigma(d_1)})\ldots \alpha_k(V_{\sigma(N-d_k+1)},\ldots,V_{\sigma(N)})
\end{align*}
where $\epsilon(\sigma)$ is the sign of the permutation $\sigma$.

\subsection{Properties of Chern-Simons forms}

\subsubsection*{Relation to Chern-Weil forms}
An easy computation shows that $d(\cs(\theta))=\Tr(\Omega\wedge\Omega)$.

\subsubsection*{Variation}
If $\theta^t$ is a $1$-parameter family of $1$-forms and $\dot{\theta}=\pl_t\theta^t|_{t=0}$, the variation of $\cs$ is computed using the trace identity:
\begin{align}\label{varcs}
\pt \cs(\theta^t)|_{t=0}=\Tr(\dtheta\wedge d\theta+\theta\wedge d\dtheta +2\dtheta\wedge\theta^2)=d\Tr(\dtheta\wedge\theta)+2\Tr(\dtheta\wedge\Omega).
\end{align}

\subsubsection*{Pull-back}
If $\Phi:Z'\to Z$ is a smooth map, we have $\cs(\Phi^*\theta)=\Phi^*\cs(\theta)$.

\subsubsection*{Action of representations}
If $\theta$ takes values in a linear Lie algebra $\mathfrak{g}\subset M_n(\cc)$ and $\rho$ is a representation of 
$\mathfrak{g}$ in $M_m(\cc)$ such that there exists some $\mu_\rho\in\cc$ with $\Tr(\rho(a)\rho(b))=\mu\Tr(ab)$ for every $a,b\in\mathfrak{g}$, then 
\begin{equation}\label{csad}
\cs(\rho(\theta))=\mu_\rho\cs(\theta).
\end{equation}

\subsubsection*{Gauge transformation}
If $a:Z\to M_n(\cc)$ is a smooth map and $\gamma:=a^{-1}\theta a+a^{-1}da$, then
\begin{equation}\label{gatr}
\cs(\gamma)=\cs(\theta)+d\Tr(\theta\wedge da a^{-1})-\tfrac13 \Tr((a^{-1}da)^3).
\end{equation}

The particular cases of $\theta$ considered below are connection $1$-forms either in a principal bundle or in a trivialization of a vector bundle.

\subsection{The $\rm PSL_2(\cc)$ invariant}

Let $\theta$ be the flat connection in the $\Psldc$ bundle $F^\cc(X)$ of an oriented hyperbolic $3$-manifold $X$, as defined in the previous section. Let $S:X\to F(X)$ be a section in the orthonormal frame bundle (recall that oriented $3$-manifolds are parallelizable), so $q\circ S$ is a section in $F^\cc(X)$ where $q$ is the natural map from $F(X)$ to $F^\cc(X)$ defined in \eqref{mapq}. The $\Psldc$ Chern-Simons form $\chs(\theta,S)$ is by definition 
the complex valued $3$-form $(q\circ S)^*\chs(\theta)$ on $X$. 

For $X$ compact, the $\Psldc$ Chern-Simons invariant of $\theta$ with respect to $S$ is defined by 
\[\CSP(\theta,S):=-\tfrac{1}{4\pi^2}\int_X \chs(\theta,S).\]
The normalization coefficient in front of $\CSP$ is so chosen because
\[-\tfrac{1}{4\pi^2}\int_{K} \Tr(-\tfrac13 (\omega_{MC})^3)=1,\]
where $\omega_{MC}$ is the $\sldc$-valued Maurer-Cartan $1$-form on $\Psldc$, and $K$ is the (compact) stabilizer of $j\in\hh^3$. Because $\Psldc=K\times\hh^3$ is homotopy equivalent to $K$, this identity implies that $\Tr(\omega_{\mathrm{MC}})/12\pi^2$ is an integer cohomology class on $\Psldc$. Thus for $X$ closed, \eqref{gatr} implies that $\CSP(\theta,S)$ is independent of $S$ modulo $\zz$.

\begin{definition}
Let $(X,g)$ be a convex co-compact hyperbolic $3$-manifold and $S$ a section in the orthonormal frame bundle which is 
even to first order. The $\Psldc$ Chern-Simons invariant of $g$ with respect to $S$ is defined by 
\[\CS^{\Psldc}(g,S):=-\tfrac{1}{4\pi^2}\int_X^0 \chs(\theta,S)=-\tfrac{1}{4\pi^2}{\rm FP}_{\eps\to 0}\int_{x>\eps}(q\circ S)^*\chs(\theta)\]
where $\theta$ is the flat connection in the $\Psldc$ bundle $F^\cc(X)$ induced by $g$. 
\end{definition}
This invariant is our main object of study in the present paper.
 
We can express $\cs((q\circ S)^*\theta)$ in terms of the Riemannian connection of $g$ as follows:
let
\begin{align*}
h_1:=\begin{bmatrix} 0&1\\1&0 \end{bmatrix},&&h_2:=\begin{bmatrix} 0&i\\-i&0 \end{bmatrix},&&h_3:=\begin{bmatrix} 1&0\\0&-1 \end{bmatrix}
\end{align*}
be a complex basis in $\sldc$. The corresponding Killing vector fields on $\hh^3$ evaluated at $j$ take the values
\begin{align*}
\kappa_{h_k}=2\partial_{y_k},&&\kappa_{ih_k}=0
\end{align*}
for $k=1,2,3$. If $U_k$ is the section over $X$ in the bundle $F^\cc(X)\times_\ad \sldc$ corresponding to the vector $h_k$ in the trivialization $q\circ S$, the above relations show that the complex vector field corresponding to
$U_k$ by the isomorphism from Proposition \ref{vbE} is just $2S_k$. Thus 
\[\ad((q\circ S)^*\theta)=\omega+iT\]
where $\omega$ is the $\sot$-valued connection $1$-form of the Levi-Civita covariant derivative $\nabla$ in the frame $S$, and $T$ denotes the $\sot$ valued $1$-form $T$ from Equation \eqref{deft} in the basis $S$.
\begin{align}\label{omet}
\omega_{ij}(Y):=g(\nabla_YS_j,S_i),&& T_{ij}(Y):=g(Y\x S_j,S_i)
\end{align}
\begin{lemma}
The $\Psldc$ Chern-Simons form of a hyperbolic metric on a $3$-manifold satisfies
\[\cs((q\circ S)^*\theta)=\tfrac{1}{4}\cs(\omega+iT).\] 
\end{lemma}
\begin{proof}
We use the following identity, valid for every $u,v\in\sldc$:
\[\Tr^{\Mdc}(uv)=\tfrac14 \Tr^{\Mtc}(\ad_u\ad_v).\]
The lemma follows from the above discussion and Equation \eqref{csad}.
\end{proof}
By this Lemma, the $\Psldc$ Chern-Simons form of $g$ in the trivialization $S$ is given by 
$\cs(\omega+iT)$ thus the $\Psldc$ Chern-Simons invariant is also given by
\[\CSP(g,S)=-\tfrac{1}{16\pi^2}\int_X^0 \chs(\omega+iT).\]

\begin{prop}\label{chsomega^c} 
The ${\rm PSL}_2(\cc)$ Chern-Simons form on a hyperbolic $3$-manifold, pulled back by a section $S\in F_0(X)$, has the following real and imaginary parts:
\[\chs(\theta,S)=2i\vol_g + \tfrac{i}{4}d(\Tr(T\wedge \omega))+
\tfrac{1}{4}\Tr(\omega\wedge d\omega+\tfrac{2}{3}\omega^3).\]
\end{prop}
\begin{proof} Since the connection $D$ is flat, it follows that $\cs(\omega+iT)=-\tfrac13 \Tr((\omega+iT)^3)$.
By the above lemma we can write (using the cyclicity of the trace)
\begin{equation}\label{calcul}\begin{split}
-12\chs(\theta,S)=&\Tr^{\Mtc}(\omega+iT)^3)\\
=&\left(\Tr^{\Mtc}(\omega^3-3T^2\wedge\omega)+
i\Tr(3\omega^2\wedge T-T^3)
\right).
\end{split}\end{equation}
The vanishing of the $D$-curvature implies $d\omega+\omega^2-T^2=0$ and $dT+T\wedge \omega+\omega\wedge T=0$.
Taking the exterior product of the first identity with $\omega$ and taking the trace we deduce that 
\begin{equation}\label{reel}
\Tr(\omega^3-3T^2\wedge\omega)=-3\Tr(\omega\wedge d\omega+\tfrac{2}{3}\omega^3)
\end{equation}
Similarly, since both $T\wedge (d\omega+\omega^2-T^2)$ and 
$\omega\wedge (dT+T\wedge \omega+\omega\wedge T)$ are $0$, one can take their trace and make the difference to deduce  
\[
0=\Tr(T\wedge d\omega - dT\wedge \omega-T^3-\omega^2\wedge T)=-d(\Tr(T\wedge \omega))-\Tr(T^3)-\Tr(\omega^2\wedge T)
\]
and then 
\begin{equation}\label{imagin}
\Tr(3\omega^2\wedge T-T^3)=-4\Tr(T^3)-3d(\Tr(T\wedge \omega)).
\end{equation}
We also easily see that $\Tr(T^3)=6\vol_g$ and therefore combining \eqref{imagin}, \eqref{reel} and \eqref{calcul}, we have proved the Proposition. 
\end{proof}


\subsection{The $\Sot$ Chern-Simons invariant}

Let $(X,g)$ be an oriented Riemannian manifold of dimension $3$.
Denote by $\cs(g)$ the Chern-Simons form of the Levi-Civita connection $1$-form of $g$ on the orthonormal frame bundle $F(X)$. Note that $S^*\cs(g)=\cs(\omega)$ where $\omega$ is the connection $1$-form in the trivialization $S$.

If $X$ is compact, for every section $S:X\to F(X)$ in the orthonormal frame bundle we define the Chern-Simons invariant
\[\CS(g,S):=-\tfrac{1}{16\pi^2}\int_X S^*\cs(g).\]

When $a:X\to\Sot$ is a compactly supported map (i.e., $a\cong 1$ outside a compact), we have 
\begin{equation}\label{intmc}
\tfrac{1}{16\pi^2}\int_X \Tr(-\tfrac13 (a^{-1}da)^3)=-\deg(a)\in\zz,
\end{equation}
so for $X$ is closed, using \eqref{gatr} we see that the Chern-Simons invariant is independent of $S$ modulo $\zz$. 

We aim to define a $\Sot$ Chern-Simons invariant associated to the Levi-Civita connection 
$\nabla^{g}$ on an  asymptotically hyperbolic $3$-manifold $(X,g)$  with totally geodesic boundary.
First, we fix a geodesic boundary defining function $x$ and we set 
\[\hat{g}:=x^2g.\]
Let $\hat{S}:X\to F(X,\hat{g})$ be a smooth section of the orthonormal frame bundle associated to the metric $\hat{g}$.
We say that $\hat{S}$ is \emph{even to first order} if $\mc{L}_{\pl_x}\hat{S}|_{M}=0$ where $\mc{L}$ denotes the Lie derivative; note that this coincides with the local definition from Section \ref{para22}. 
We define, starting from $\hat{S}$, a section $S:=x\hat{S}$ in the frame bundle $F_0(X)$ associated to $g$. 

\begin{definition}
The ${\rm SO}(3)$ Chern-Simons invariant of an asymptotically hyperbolic metric $g$ with totally geodesic boundary, 
with respect to an even to first order trivialization $S$ of $F_0(X)$, is 
\begin{align*}
\CS(g,S):=-\tfrac{1}{16\pi^2}\int^0_X {S}^*\cs(g).
\end{align*}
\end{definition}

\section{Comparison of the asymptotically hyperbolic and compact Chern-Simons $\Sot$
invariants}

For any pair of conformal metrics $\hat{g}=x^2 g$ we can relate $\cs(g,S)$ to $\cs(\hat{g},\hat{S})$ as follows. We denote by $\omega, \hat{\omega}$ the connection $1$-forms of $g,\hat{g}$ in the trivialization $S$, respectively $\hat{S}=x^{-1}S$. 
For every $Y\in T{X}$ that means
\begin{align*}
\hat{\omega}_{ij}(Y)=\hat{g}(\nabla^{\hat{g}}_Y\hat{S}_j,\hat{S}_i),&& 
\omega_{ij}(Y)=g(\nabla^{g}_YS_j,S_i).
\end{align*}
\begin{lemma}\label{cccon}
The connection forms of the conformal metrics $g$ and $\hat{g}=x^2g$ satisfy $\hat{\omega}=\omega +\alpha$, where
\begin{align}\label{omegavs'}
\alpha_{ij}(Y):=\hat{g}(Y,\hat{S}_i) \hat{S}_j(a)-\hat{g}(Y,\hat{S}_j) \hat{S}_i(a)=g(Y,S_i)S_j(a)-g(Y,S_j)S_i(a)
\end{align}
with $a:=\log(x)$.
\end{lemma}
\begin{proof}
An easy computation using Koszul's formula for the Riemannian connection in a frame.
\end{proof}

Let $g^t=x^{2t}g$ where $t\in[0,1]$, then $S^t:=x^{-t}S$ defines a section of the frame bundle 
$F^t(X)$ of $g^t$. Consider $\omega^t$ the connection form of $g^t$ in the basis $S^t$, and
write $\alpha^t=\omega^t-\omega$. Notice from \eqref{omegavs'} that $\alpha^t=t\alpha$ is linear in $t$, so
we compute the variation of $\cs(\omega^t)$ using \eqref{varcs}:
\[\pl_t \cs(\omega^t)|_{t=0}=d\Tr(\alpha\wedge \omega)+2 \Tr(\alpha\wedge\Omega)\]
where $\Omega=d\omega+\omega\wedge \omega$ is the curvature of $\omega$. 
\begin{lemma}
We have $\Tr(\alpha\wedge \Omega)\equiv 0$.
\end{lemma}
\begin{proof}
At points where $\nabla a=0$ this is clear. 
At other points, take an orthonormal basis $(X_1,X_2,X_3)$ of $TX$ 
for $g$ such that $X_3$ is proportional to $\grad(a)$. Since $\alpha\wedge\Omega$ is a tensor, we can compute its trace in the basis $X_j$ instead of $S_j$:
\[\begin{split}
\Tr(\alpha\wedge \Omega)(X_1,X_2,X_3)= &\sum_{i,j}\alpha_{ij}(X_1)\Omega_{ji}(X_2,X_3)-\alpha_{ij}(X_2)\Omega_{ji}(X_1,X_3)\\
=&2 (\cjg R_{X_2 X_3}X_1,\grad(a)\cjd -\cjg R_{X_1 X_3}X_2,\grad(a)\cjd)
\end{split}\]
and this vanishes using the symmetry of the Riemannian curvature together with the fact that 
$X_3$ and $\grad(a)$ are collinear. 
\end{proof}
We deduce that $\pl_t \cs(\omega^t)|_{t=0}=d\Tr(\alpha\wedge \omega)$ and so
\[\pl_t\cs(\omega^t)=\pl_s \cs(\omega^{t+s})|_{s=0}
=d(\alpha^t\wedge \omega^t)=d\Tr(\alpha\wedge (\omega+t\alpha)).\]
Since $\Tr(\alpha\wedge \alpha)=0$ by cyclicity of the trace, we find
\begin{align} \label{csformcf}
\cs(\hat{g},\hat{S})= \cs(g,S)+d\Tr(\alpha\wedge \omega).
\end{align}

\begin{prop}\label{CSconforme}
Let $g$ be an asymptotically hyperbolic metric on $X$ with totally geodesic boundary, let $x$ be a smooth geodesic boundary defining function and set $\hat{g}:=x^2g$.
Let $\hat{S}$ be an even to first order section in $F(X)$ with respect to $\hat{g}$, and let $S=x\hat{S}$
be the corresponding section in $F_0(X)$. Then the $\Sot$ Chern-Simons invariants of $g$ and $\hat{g}$ with respect to $S,\hat{S}$ coincide: 
\[\CS(g,S)=\CS(\hat{g},\hat{S}).\]
\end{prop}
\begin{proof}
By integration on $X$ and using Stokes, we get
\begin{equation}\label{CSChatS}
16\pi^2\CS(\hat{g},\hat{S})=16\pi^2\CS(g,S)-{\rm FP}_{\eps\to 0}\int_{x=\eps}\Tr(\alpha\wedge \omega).
\end{equation}
The proof is finished by showing that the trace $\Tr(\alpha\wedge \omega)$ is odd in $x$ to order $O(x^4)$, so
\[{\rm FP}_{\eps\to 0}\int_{x=\eps}\Tr(\alpha\wedge \omega)=0.\]
For this, note that $x\alpha$ is smooth in $x$ and has an even expansion at $x=0$ in powers of $x$ up to $O(x^3)$ by assumption on the section $S$, while $x\omega$ is smooth in $x$ but a priori not even. 
Setting $a:=\log x$ we write for $Y_1,Y_2$ vector fields on $\pl\oX$ (thus orthogonal to $\nabla a$) and 
$\cjg\cdot,\cdot\cjd:=g(\cdot,\cdot)$:
\begin{align*}
\lefteqn{\Tr(\alpha\wedge \omega)(Y_1,Y_2)}\\
&=\sum_{1\leq i,j\leq 3}\alpha_{ij}(Y_1)\omega_{ji}(Y_2)-\alpha_{ij}(Y_2)\omega_{ji}(Y_1)\\
&= \sum_{1\leq j\leq 3} S_j(a)\cjg \nabla_{Y_2}S_j,Y_1\cjd-S_j(a)\cjg \nabla_{Y_1}S_j,Y_2\cjd
-\cjg S_j,Y_1\cjd\cjg \nabla_{Y_2}S_j,\nabla a\cjd+\cjg S_j,Y_2\cjd\cjg \nabla_{Y_1}S_j,\nabla a\cjd\\
&=2\cjg \nabla_{Y_2}\nabla a,Y_1\cjd -2\cjg \nabla_{Y_1}\nabla a,Y_2\cjd 
-2\sum_{1\leq j\leq 3} Y_2(S_j(a))\cjg S_j,Y_1\cjd- Y_1(S_j(a))\cjg S_j,Y_2\cjd\\
&=\sum_{1\leq j\leq 3} -2Y_2(S_j(a))\cjg S_j,Y_1\cjd+2 Y_1(S_j(a))\cjg S_j,Y_2\cjd 
\end{align*}
and this is odd in $x$ to order $O(x^2)$ since $da=-\frac{dx}{x}$, $\hat{S}_j=x^{-1}S_j$ is even to first order, and the metric has totally geodesic boundary (i.e. $x^2g$ is even to order $O(x^3)$). 
\end{proof}

\section{Comparison of the $\Psldc$ and $\Sot$ invariants in the hyperbolic setting}

In this section we establish the relation between the ${\rm PSL}_2(\cc)$ and the ${\rm SO}(3)$ Chern-Simons invariants. This was known in the compact case and in the finite volume case since the work of Yoshida \cite{Yo}.
\begin{prop}\label{comparaisons}
Let $(X,g)=\Gamma\backslash \hh^3$ be a convex co-compact hyperbolic $3$-manifold with $\Gamma\subset {\rm PSL}_2(\cc)$ and let $\theta$ be the associated flat connection on the bundle $F^\cc(X)=\hh^3\x_\Gamma {\rm PSL}_2(\cc)$. Let $S:X\to F_0(X)$ be an even section of $F_0(X)$. Then
\[\CSP(\theta,S)=-\tfrac{i}{2\pi^2}{\rm Vol}_R(X)+\tfrac{i}{2\pi}\chi(M)+\CS(g,S).\]
\end{prop}
\begin{proof}
Using Proposition \ref{chsomega^c} and Stokes's formula, we have 
\begin{align*}
\CSP(\theta,S)=&{\rm FP}_{\eps\to 0}\int_{\{x>\eps\}}-\tfrac{i}{2\pi^2}\vol_{\hh^3}+\tfrac{i}{16\pi^2}d(\Tr(\omega\wedge T)) -\tfrac{1}{16\pi^2} \cs(\omega)\\
=&-\tfrac{i}{2\pi^2}{\rm Vol}_R(X)-{\rm FP}_{\eps\to 0}\tfrac{i}{16\pi^2}\int_{x=\eps}\Tr(\omega\wedge T)+\CS(g,S).
\end{align*}
The conclusion follows from Lemma \ref{lemot}.
\end{proof}

\begin{lemma}\label{lemot}
We have
\begin{equation}\label{TromegaT}
{\rm FP}_{\eps\to 0}\int_{x=\eps}\Tr(T\wedge \omega) =2\int_M \scal_{h_0}\vol_{h_0}=8\pi \chi(M).
\end{equation}
\end{lemma}
\begin{proof}
Let $U_j:=x^{-1} S_j$ denote the orthonormal frame for the compact metric $\hat{g}=x^2 g$, 
$(S^1,S^2,S^3)$ the dual basis to $S$, $\hat{\omega}_{ij}(Y):=\hat{g}(\nabla^{\hat{g}}_{Y} U_i, U_j)$ the Levi-Civita connection $1$-form of $\hat{g}$ in the frame $U$, and $U^j=xS^j$ the dual co-frame. Let $Y_1,Y_2$ be a local orthonormal frame on $M$ for $h_0$ of eigenvectors for the map $A$ defined on $TM$ by \eqref{metrhyp}, extended on $X$ constantly in $x$ near $M$.

We split $\omega=\homega-\alpha$ and we first compute 
\begin{align*}
\Tr(T\wedge\alpha)(Y_1,Y_2)=&\sum_{i,j}\lan Y_1\times S_i, S_j\ran(\lan Y_2,S_j\ran S_i(a)-\lan Y_2,S_i\ran S_j(a))\\
&-\lan Y_2\times S_i, S_j\ran(\lan Y_1,S_j\ran S_i(a)-\lan Y_1,S_i\ran S_j(a))\\
=&-4\lan Y_1\times Y_2,\nabla a\ran.
\end{align*}
Define $\tY_j:=(1+\tfrac12 x^2A)^{-1}Y_j$. Then $x\tY_1,x\tY_2,\nabla a$ form an orthonormal frame near $x=0$
and $x\tY_1\times x\tY_2=\nabla a.$
Thus, $x^2(1+\tfrac{x^2\lambda_1}{2})^{-1}(1+\tfrac{x^2\lambda_2}{2})^{-1}\lan Y_1\times Y_2,\nabla a\ran=1$
which shows that 
\[{\rm FP}_{\eps\to 0} \Tr(T\wedge\alpha) = 2\tr(A) \vol_{h_0}.\]

Let us now compute the form $\Tr(\homega\wedge T)$ on the hypersurface $x=\eps$. Notice that
\begin{align*}
T_{12}=S^3,&&T_{23}=S^1,&&T_{31}=S^2
\end{align*}
so $xT$ is smooth in $x$, and we easily see that 
\[\demi \Tr(\homega \wedge T)=S^1\wedge\homega_{23}+S^2\wedge \homega_{31}+S^3\wedge \homega_{12}.\]
Using Koszul formula and the evenness of $g$ and $S$, for a vector $Y\in TM$ independent of $x$ the term $\hat{\omega}_{ji}$ can be decomposed under the form
\[\begin{split}
2\homega_{ji}(Y)=& U_i(\hat{g}(Y, U_j))- U_j(\hat{g}(Y, U_i))-\hat{g}([ U_i, U_j],Y)
+\text{ even function of $x$}\\
=&dY^\#(U_i,U_j)
+\text{ even function of $x$}
\end{split}\] 
so the odd component is tensorial in $U_j$. Therefore we can compute ${\rm FP}_{\eps\to 0}\Tr(\homega\wedge T)$ 
using the orthonormal frame $\tY_1,\tY_2,\tY_3:=\px$:
\begin{align*}
\Tr(T\wedge \homega)(Y_1,Y_2)=&x^{-1}\sum_{i,j} \lan Y_1\times \tY_i,\tY_j\ran \lan\nabla_{Y_2} \tY_i,\tY_j\ran
- \lan Y_2\times \tY_i,\tY_j\ran \lan\nabla_{Y_1} \tY_i,\tY_j\ran\\
=&x^{-1}\sum_i\lan\nabla_{Y_2}\tY_i,Y_1\times \tY_i\ran -\lan\nabla_{Y_1}\tY_i,Y_2\times \tY_i\ran.
\end{align*}
(here the vector product is with respect to $\hat{g}$). Since $\tY_j-Y_j$ is of order $x^2$, the finite part is unchanged if we replace $Y_1,Y_2$ by $\tY_1,\tY_2$ in the above, thus getting
\[x^{-1}\sum_i\lan\nabla_{\tY_2}\tY_i,\tY_1\times \tY_i\ran -\lan\nabla_{\tY_1}\tY_i,\tY_2\times \tY_i\ran.\]
For $k=1,2$ the coefficient of $x$ in $\lan\nabla_{\tY_k} \tY_k,\px\ran=-\lan\nabla_{\tY_k}\px, \tY_k\ran$
is $-\lambda_k$. We therefore get
\[{\rm FP}_{\eps\to 0}\Tr(T\wedge \homega)=-2\tr(A)\vol_{h_0}.\]
Together with the identity $2\tr(A)=-\scal_{h_0}$ and Gauss-Bonnet this ends the proof.
\end{proof}

\section{The Chern-Simons line bundle and its connection}

\subsection{The tangent space of Teichm\"uller space as the set of hyperbolic funnels}\label{teichmuller}

In this subsection, we shall see that the tangent space $T\mc{T}$ of Teichm\"uller space of Riemann surfaces of genus $\bf{g}$ can be identified with ends of hyperbolic $3$-manifolds of funnel type. For Teichm\"uller space definition and conventions, we follow the book of Tromba \cite{Tr}.

The Teichm\"uller space $\mc{T}$ is defined here as the 
quotient $\Mm(\Sigma)/\mc{D}_0(\Sigma)$ where $\Mm(\Sigma)$ is the set of metrics with Gaussian curvature $-1$ on a fixed smooth surface $\Sigma$ of genus $\bg$, and $\mc{D}_0(\Sigma)$ is the group of orientation-preserving diffeomorphisms of $\Sigma$ which are isotopic to the identity.  Here $M$ is not necessarily connected and 
$\bg\in (\nn\setminus\{0,1\})^{N}$ where $N=\pi_0(M)$. 

First, we shall identify each point of $T\mc{T}$ with an isometry class of $3$-dimensional hyperbolic ends, 
with conformal infinity given by the base point. 
\begin{definition}\label{hypfune}
A \emph{hyperbolic funnel}  is a couple $(M,g)$ where $M$ is a Riemann surface  (not necessarily connected)
equipped with a  metric $h_0$ of Gaussian curvature  $-1$
and $g$ is a metric on the product $M\x(0,\eps)_x$ for some small $\eps>0$, which is of the form 
\begin{align}\label{funnelform}
 g= \frac{dx^2+h(x)}{x^2}, && h(x):=h_0+x^2h_2+\tfrac{1}{4}x^4h_2\circ h_2
 \end{align}
where $h_2$ is a symmetric tensor satisfying 
\begin{align*}
\Tr_{h_0}(h_2)= \kappa, && {\rm div}_{h_0}(h_2)=0.
 \end{align*}
\end{definition}
It is shown in Fefferman-Graham \cite{FefGr} that $M\x (0,\eps)$ equipped with such a metric $g$ is a (non-complete) hyperbolic manifold if $\eps>0$ is chosen small enough, and conversely every end of a convex co-compact hyperbolic manifold with conformal infinity $(M,\{h_0\})$ and ${\rm genus}(M)>1$  is isometric to a unique 
funnel \eqref{funnelform} with $h_0$ the hyperbolic metric representing the conformal class $\{h_0\}$. There is an action of the group $\mc{D}$ of diffeomorphisms of $M$ on the space of funnels, simply given by 
\[ \psi^*(M,g):= \left(M,  \frac{dx^2+\psi^*h(x)}{x^2}\right)\]
for all $\psi\in\mc{D}$, where $\psi^*h(x)$ is the pull-back of the metric $h(x)$ on $M$.
Notice also that a funnel induces a representation of $\pi_1(M)$ into $\Psldc$ up to conjugation.

The tangent space $T\Mm$ has a natural inner product, the $L^2$-metric, defined as follows (see \cite[Sec. 2.6]{Tr}): let  $h_0\in \Mm$, and  $h,k\in T_{h_0}\Mm$. Since $\Mm$
is a Fr\'echet submanifold in the space of symmetric tensors on $M$, it follows that $T\Mm\subset S^2(T M)$; define
\begin{equation}\label{metwpm}
 \cjg h,k\cjd:= \int_{M} \lan h,k\ran_{h_0} {\rm dvol}_{h_0}.
 \end{equation}
This scalar product is $\Diff(M)$-invariant. 

For any $h_0\in\Mm(M)$ consider the vector space
\[V_{h_0}:=\{ h \in C^\infty(S^2T^*M); \Tr_{h_0}(h)=0; {\rm div}_{h_0}(h)=0\}, \]
i.e., the set of transverse traceless symmetric tensors with respect to $h_0$. This is a real vector space
of finite dimension which is precisely the orthogonal complement in $T_{h_0}\Mm$ of the orbit of $\cD_0$ with respect to the $L^2(M,h_0)$ inner product.
When $h_0$ varies, these spaces form a locally trivial vector bundle $V$ over $\Mm(M)$ of rank
$6\bg-6$ (assuming that the genera of the connected components $M_j$ are strictly larger than $1$), which we think of as the horizontal tangent bundle in the principal Riemannian fibration $\Mm(M)\to \Tei$.
The group $\Diff(M)$ acts isometrically on this bundle by pull-back of tensors, and the restriction of this action to the subgroup $\cD_0(M)$ is free. The quotient of $V$ by $\cD_0$ is identified \cite[Sec 2.4]{Tr} with the tangent bundle $T\mc{T}_{\bg}$ of the Teichm\"uller space of genus $\bg$. Thus, Teichm\"uller space inherits a Riemannian metric called the \emph{Weil-Petersson metric}. Explicitly, on vectors in $T_{[h_0]}\Tei$ described by trace-free, divergence free symmetric tensors $h,k$ with respect to a representative $h_0\in[h_0]$, the Weil-Petersson metric is defined by
\begin{equation}\label{metric}
 \cjg h,k\cjd_{\rm WP}:= \int_{M} \lan h,k\ran_{h_0} {\rm dvol}_{h_0}.
 \end{equation}

The following is a direct consequence of the discussion above:
\begin{lemma}\label{identiffunhyp}
There is a canonical bijection $\Psi$ from the total space of the horizontal tangent bundle
$V\to \Mm(M)$ to the set $F_{\bg}$ of hyperbolic funnels of genus $\bg$, defined explicitly by
\begin{align}\label{map}
\Psi: (h_0, h^0_2) \mapsto \left(M, \tfrac{dx^2+h(x)}{x^2}\right),&& h(x)=h_0+x^2h_2+\tfrac{x^4}{4} h_2\circ h_2, && h_2=h^0_2+\tfrac{h_0}{2}.
\end{align}
This bijection commutes with the action of $\mc{D}_0$ on both sides and hence descends to a bijection
from $T\Tei$ to the space of $\cD_0(M)$-equivalence classes of hyperbolic funnels.
\end{lemma} 

Any divergence free traceless tensor $k=udx^2-udy^2-2vdxdy$ with respect to a metric $h_0$ is the real part of a quadratic holomorphic differential (QHD in short)
\[\demi k={\rm Re}(k^{0,1}) \,\,\textrm{ with }k^{0,1}:= \demi (u+iv)dz^2 \textrm{ in local complex coordinates }z=x+iy.\]
The complex structure $J$ on Teichm\"uller space is then given by multiplication by $-i$ on QHD, which on the level of transverse traceless tensors means  
\begin{equation}\label{complexstruct} 
Jk:= vdx^2-vdy^2+2udxdy 
\end{equation}
or setting $K$ to be the symmetric endomorphism of $TM$ defined by $k(\cdot,\cdot)=h_0(K\cdot,\cdot)$, 
\[JK= \left(\begin{array}{cc}
0 & - 1\\
1 & 0
\end{array}\right)\left(\begin{array}{cc}
u & - v\\
-v & -u
\end{array}\right)=\left(\begin{array}{cc}
v & u\\
u & -v
\end{array}\right).\]
The space $T^{0,1}\Tei$ is then defined to be the subspace of complexified tangent space $ T_\cc\Tei$
spanned by the elements $k+iJk$ with $k\in T\Tei$, and $T^{1,0}\Tei$ is spanned
by the $k-iJk$.  Notice also that, with the notations used just above, one has
\begin{equation}
k^{0,1}= \demi(k+iJk)\in T^{0,1}\Tei, \quad  k^{1,0}=\bbar{k^{0,1}}=\demi(k-iJk)\in T^{1,0}\Tei. 
\end{equation}

The Weil-Petersson metric on $T\Tei$ induces an isomorphism $\Phi$ between $T\Tei$ and $T^*\Tei$. It is also a Hermitian metric for the complex structure $J$, in the sense that $\cjg Jh,Jk\cjd_{\rm WP}=\cjg h,k\cjd_{\rm WP}$ for all $h,k\in T\Tei$, the associated symplectic form is $\omega_{\rm WP}(\cdot,\cdot):=\cjg J\cdot,\cdot\cjd_{\rm WP}$. By convention, the metric $\cjg \cdot,\cdot\cjd_{\rm WP}$ on $T\Tei$ is extended to be 
bilinear on $T_\cc\Tei$, so that $\cjg k^{0,1},h^{0,1}\cjd_{\rm WP}=\cjg k^{1,0},h^{1,0}\cjd_{\rm WP}=0$ for all 
$h,k\in T\Tei$ and $\cjg k^{0,1},k^{1,0}\cjd_{\rm WP}\geq 0$ for all $k\in T\Tei$.

On $T^*\Tei$, there is a natural symplectic form, obtained by taking the exterior derivative $d\mu$ of the Liouville 
$1$-form  $\mu$ defined for $h_0\in \Tei$, $k^*\in T_{h_0}^*\Tei $ by 
\[\mu_{(h_0,k^*)}:= k^*. d\pi  \]
if $\pi:T^*\Tei \to \Tei$ is the natural projection. Since $T^*\Tei$ has a complex structure induced naturally by that of $\Tei$ we can also define the $(0,1)$ component $\mu^{1,0}$ of the Liouville measure. 
The Liouville form $\mu$ and $\mu^{1,0}$  pull-back to natural form  
on $T\Tei$ through $\Phi$, satisfying 
\[\Phi^*\mu_{(h_0,k)}(\dot{h}_0,\dot{k})= \cjg k,\dot{h}_0\cjd_{\rm WP}, \quad 
\Phi^*\mu^{1,0}_{(h_0,k)}(\dot{h}_0,\dot{k})= \cjg k^{0,1},\dot{h}^{1,0}_0\cjd_{\rm WP}\]  
for  $(h_0,k)\in T\Tei$, and $(\dot{h}_0,\dot{k})\in TT_{h_0}\Tei=T_{h_0}\Tei\oplus T_{h_0}\Tei.$
Notice that $d\mu^{1,0}$ is a $(1,1)$ type form on $T^*\mc{T}_{\bg}$.

\subsection{The cocycle}

In order to define the Chern-Simons line bundle $\mc{L}$ over $\mc{T}_{\bf g}$ in a way similar to Freed \cite{Freed} and Ramadas-Singer-Weitsman \cite{RSW}, we need to define a certain cocycle. The natural bundle 
turns out to be the $\Sot$ Chern-Simons line bundle associated to a $3$-manifold bounding a given surface. 

For each $h_0\in \Mm$, let $(\oX,\hat{g})$ be a compact Riemannian manifold with totally geodesic boundary $(M,h_0)$. 
Consider the map $c^X:\cun(M,F(X))\times \cun_\ext(M,\Sot)\to\cc$ defined by
\[c^X(\hat{S},a)=\exp\left(2\pi i\int_{M} \tfrac{1}{16\pi^2}\Tr(\hat{\omega}\wedge da\, a^{-1})
+2\pi i\int_X\tfrac{1}{48\pi^2} \Tr((\til{a}^{-1}d\til{a})^3)\right)\]
where $\hat{\omega}$ is the connection form of the Levi-Civita connection of  
$\hat{g}$ in any extension in $\cun(\oX,F(X))$ of $\hat{S}$ to $\oX$, and $\til{a}$ is any smooth extension of $a$ on $\oX$ ($\cun_\ext(M)$ means sections which are extendible to $\oX$). Note that any 
$a\in\cun(M,\Sot)$ can be extended to some $\til{a}$ on a handlebody with boundary $M$. 
This definition is consistent by the following Lemma.
\begin{lemma}\label{cocycle}
The value $c^X(\hat{S},a)$ depends only on $h_0$, $\hat{S}$ and  $a$.
\end{lemma}
\begin{proof} In other words, we must prove that $c^X(\hat{S},a)$ does not depend on the choice of the 
even metric $\hat{g}$, on the choice of 
$X$ bounding  $M$ and on the extensions of $(\hat{S},a)$ from $M$ to $X$.
The independence of $\hat{\omega}|_{TM}$ with respect to $\hat{g}$ and the extension of $\hat{S}$ is a consequence of Koszul formula  and the evenness to first order of $\hat{g}$. 
The independence with respect to the choice of $\bbar{X}$ and the extension $\til{a}$ is a consequence of the fact that 
\[\exp\left(-2\pi i\int_Z\tfrac{1}{48\pi^2} \Tr((\til{a}^{-1}d\til{a})^3)\right)=1\]
if $Z$ is a compact manifold without boundary and $\til{a}\in \cun(Z,\Sot)$ (see \eqref{intmc}).
\end{proof}

We therefore get a map $c^X:\cun(M,F(X))\times \cun_\ext(M,\Sot)\to\cc$ by associating
to $(\hat{S}_{|M},a|_{M})$ the quantity $c^X(\hat{S},a)$, where the subscript $\ext$ denotes objects on $M$ extendible to $X$. As we shall see below, when acting on $\hat{S}\in \cun_\ext$, this map satisfies the cocycle condition
\begin{equation}\label{cocy}
c^X(\hat{S},ab)=c^X(\hat{S},a)c^X(\hat{S}a,b).
\end{equation}

\subsection{The Chern-Simons line bundle $\mc{L}$}\label{linebundle}
We follow the presentation given in the book of Baseilhac \cite{Bas}, but adapted to our setting.
\begin{definition}
The complex line $L^X_{h_0}$ over $h_0\in \Mm$ is defined for a choice of extension $X$ by 
\[L^X_{h_0}:=\{f: \cun_\ext(M,F(X))\to \cc;\, \forall a\in \cun_\ext(M,\Sot),\, f(\hat{S}a)=c(\hat{S},a) f(\hat{S})\}.\]
We define the Chern-Simons line bundle (as a set) over $\Mm$ by 
\[\mc{L}^X:=\bigsqcup_{h_0\in \Mm}L^X_{h_0}.\]
\end{definition}

Using the Gauge transformation law \eqref{gatr}, we deduce 
\begin{lemma}\label{CSisasection}
For any metric $\hat{g}$ on $\oX$ with $\hat{g}|_{\pl\oX}=h_0$, 
the map $\hat{S}\mapsto e^{2\pi i\CS(\hat{g},\hat{S})}$ is an element of the fiber over 
$h_0$. 
\end{lemma}
This fact directly implies the cocycle condition \eqref{cocy}.

An element in $L^X_{h_0}$ is determined by its value on any frame extendible to $\oX$ by the
condition $f(\hat{S}a)=c(\hat{S},a) f(\hat{S})$, therefore the dimension of the $\cc$-vector space $L^X_{h_0}$ is $1$.
If $X_1$ and $X_2$ are two fillings of $M$, 
let $Z=X_1\cup X_2$ be the oriented closed manifold obtained by gluing $X_1$ and $X_2$ along $M$, then 
any frame on $Z$ restricted to $M$ is extendible both to $X_1$ and $X_2$. For such a frame $\hat{S}$, we define an isomorphism between $\mc{L}^{X_1}$ and $\mc{L}^{X_2}$ by setting 
\begin{align*}
f^{X_1}\mapsto f^{X_2},  &&  f^{X_2}(\hat{S}):=f^{X_1}(\hat{S}).
\end{align*} 
By Lemma \ref{cocycle}, this isomorphism is independent of the choice of $\hat{S}$ extendible to $X_1$ and $X_2$. 
Therefore we have a well defined bundle $\mc{L}$ over $\Tei$ independent of the filling $X$.

We define the smooth structure on $\mc{L}$ through global trivializations as follows: 
let $\hat{S}$ be a smooth positively oriented frame (not a priori orthonormal) on $\oX$ and let 
$\hat{S}_{h_0}$ be the orthonormal frame obtained from $\hat{S}$ by Gram-Schmidt process 
with respect to the metric $dx^2+h_0$ near the boundary $\pl\oX$ and define a global trivialization by
\begin{align*} 
\mc{L} \to \Mm \x\cc  && (h_0,f )\mapsto (h_0, f(\hat{S}_{h_0})). 
\end{align*}
Changes of trivializations corresponding to different choices of $\hat{S}$ are smooth on $\Mm$, 
thus we get a structure of smooth line  bundle  on $\mc{L}$  over $\Mm$.
 
 \subsection{Action by the mapping class group}\label{actionby}
 
The mapping  class group ${\rm Mod}$ is the set of isotopy classes of orientation preserving 
diffeomorphisms of $M=\pl\oX$, it acts on $\Tei$ properly discontinuously.  
By Marden \cite[Theorem 3.1]{Marden}, the subgroup ${\rm Mod}_X$ of ${\rm Mod}$ arising from elements which extend to diffeomorphisms of $\oX$ homotopic to the identity on $\oX$ acts freely on $\Tei$ 
and the quotient $\mc{T}_X:=\Tei/{\rm Mod}_X$ 
is a complex manifold of dimension $3|\bg|-3$. Moreover the Weil-Petersson metric descends to $\mc{T}_X$.
 
Every diffeomorphism $\psi:\oX\to \oX$ induces an isomorphism $L^X_{h_0}\to L^{X}_{\psi^*h_0}$
by $f\mapsto f_\psi:=(\hat{S}\mapsto f(\psi_*\hat{S}))$. In particular since any $\psi\in \mc{D}_0$ can be extended on $\oX$
as a diffeomorphism and the map $L^X_{h_0}\to L^{X}_{\psi^*h_0}$ does not depend on the extension, 
the bundle $\mc{L}_{\bg}^X$ descends to $\Tei$ as a complex line bundle. 

We define the pull-back bundle $\pi^*\mc{L}$ on $T\Tei$ if $\pi:T\Tei\to \Tei$ is the projection on the base and we shall use the notation $\mc{L}$ instead of $\pi^*\mc{L}$.

In what follows, we shall work with Teichm\"uller space but all constructions are ${\rm Mod}_X$ invariants and descend to $\mc{T}_X$.

\subsection{Hermitian metric on $\mc{L}$}
Since the cocycle is of absolute value $1$, 
there exists on $\cL_\bg$ a canonical Hermitian metric, denoted $\cjg\cdot,\cdot\cjd_{\CS}$, given simply by 
\begin{equation}\label{metricCS} 
\cjg f_1,f_2\cjd_{\CS}:= f_1(\hat{S})\bbar{f_2(\hat{S})}
\end{equation}
if $f_1,f_2$ are two sections of $\mc{L}$ and $\hat{S}\in C^\infty(M,F(X))$.

\subsection{The connections on $\mc{L}$}\label{connectionsL}
We define $2$ different connections on $\mc{L}$.  We start with a Hermitian connection coming from the base $\Tei$.  
\begin{definition}
Let $h_0^t\in \Mm$ for $t\in\rr$ be a curve of hyperbolic metrics on $M$ extended evenly to first order 
to a metric $\hat{g}^t$ on  $X$, with $h_0^0=:h_0$. 
For any $\hat{g}$-orthonormal frame $\hat{S}$ on $X$ 
we define $\hat{S}^t$ to be the parallel transport of $\hat{S}$ in the $t$ direction with respect to the metric $\hat{G}:=dt^2+\hat{g}^t$ ($\hat{S}^t$ is then a $\hat{g}^t$-orthonormal frame). 
For any section $f$ of $\mc{L}$, we define  for $\dot{h}_0=\pl_th_0^t|_{t=0}\in T_{h_0}\Mm$ 
\[(\nabla^{\mc{L}}_{\dot{h}_0}f) (\hat{S}):= \pl_t f(h_0^t,\hat{S}^t)|_{t=0}
-2\pi if(\hat{S})\int_{M}\tfrac{1}{16\pi^2}\Tr(\dot{\hat{\omega}}\wedge\hat{\omega})\]
where $\dot{\hat{\omega}}=\pl_t{\hat{\omega}}^{t}|_{t=0}$ and ${\hat{\omega}}^{t}$ is the Levi-Civita connection $1$-form ($\sot$-valued) of the metric ${\hat{g}}^{t}$ in the frame ${\hat{S}}^{t}$.
\end{definition}
One can check that the frame $\hat{S}^t$ constructed above is even to first order. We leave this verification to the reader, it follows from the Koszul formula by writing the parallel transport equation as a system of ODEs and using the evenness of $\hat{g}^t$ and $\hat{S}^0$. 

This connection is $\Diff(\oX)$ invariant (recall that $\Diff(\oX)$ acts on $\mc{L}$ over $\Mm$), thus 
we get a connection in the Chern-Simons bundle over $\Tei$ and any of its quotients 
by a subgroup of the mapping class group acting freely on $\Tei$ whose elements can be realized as 
diffeomorphism of $\oX$ for some given $\oX$ bounding $M$. 

A straightforward application of Koszul formula shows the 
\begin{lemma}\label{paralleltr}
Let $S\in C^\infty(\oX,F_0(X))$ be an even to first order orthonormal frame on $X$ with respect to an  
even to first order AH metric $g$, and let $g^t$ be a curve of even to first order AH metrics 
with $g^0=g$. Write the metric $g^t$ near the boundary under the funnel form \eqref{funnelform} 
\[g^t=\frac{dx^2+h^t(x)}{x^2}.\]
Then the parallel transported frame $S^t$ of $S$ in the 
$t$ direction with respect to the metric $G=dt^2+g^t$ is equal to $x\hat{S}^t$ where 
$x$ is a geodesic boundary defining function for $g$ and $\hat{S}^t$ is the parallel transported 
frame of $\hat{S}:=x^{-1}S$ for $\hat{G}^t=dt^2+x^2g^t$ in the $t$-direction. 
\end{lemma}

\subsection{The curvature of $\nabla^{\mc{L}}$}

Consider the trivial fibration $\Mm\times M\to \Mm$ with fiber type $M$, with metric $h$ along the fiber
above $h\in\Mm$. The action of the group $\Diff(M)$ on $\Mm$ extends isometrically to the fibers, thus by quotienting through the free action of $\Diff_0$ we obtain the so-called universal curve $\cF\to\Tei$ with fiber type $M$, which
is a Riemannian submersion over $\Tei$. In the proof below we shall consider the restriction of the fibration
$\Mm\times M\to \Mm$ above the image of a local section in $\Mm\to\Tei$. The resulting trivial fibration is canonically diffeomorphic to an open set in $\cF$ but not isometric, although the identification is an isometry along the fibers.

\begin{prop}\label{curvcslb}
The curvature of $\nabla^{\mc{L}}$ equals $\frac{i}{8 \pi}\omega_{\rm WP}$, where $\omega_{\rm WP}$ denotes the Weil-Petersson symplectic form on $\Tei$, $\omega_{\rm WP}(U,V)=\lan JU,V\ran_\WP$.
\end{prop}
\begin{proof}
Let $f$ be a local section in $\cLg\to\cU\subset\Teig$ constructed as follows:
first, choose a local section $s:\cU\subset\Teig\to \Mm$ in the principal fibration $\Mm\to\Teig$, i.e., a smooth family of hyperbolic metrics $\cU\ni [h]\mapsto h$ which by projection give a local parametrization of $\Teig$. By restricting the metric of $\Mm\times M$ to $s(\cU)\times M=:\cMU$, 
we obtain a metric on $\cMU$ with respect to which $s(\cU)$ and $M$ are orthogonal, the projection on $\cU$ is a Riemannian submersion on the Weil-Petersson metric \eqref{metwpm} on $s(\cU)$, and the metric on the fiber $\{h\}\times M$ is $h=s([h])$. Next, extend each metric $h\in s(\cU)$ to a metric $g_{[h]}$ on a fixed compact manifold $X$ with boundary $M$, so that for each $[h]\in\cU$, $g_{[h]}$ restricts to $h$ on $M$, has totally geodesic boundary, and depends smoothly on $[h]$. We get in this way a metric $G$ on $\cX^\cU:=s(\cU) \times X$ with respect to which $s(\cU)$ and $X$ are orthogonal, the projection on $\cU$ is a Riemannian submersion, and the metric on the fiber $\{h\}\times X$ is $g_{[h]}$.
Define
\begin{align*}
f:\cU\to \cLg,&& f([h]):=e^{2\pi i\CS(g_{[h]},\cdot)}.
\end{align*}
Let $\rr\ni t\mapsto h^t$ be a smooth curve in $s(\cU)$ 
parametrized by arc-length 
and $\dot{h}$ its tangent vector at $t=0$. By the variation formula \eqref{varcs}, the covariant derivative of the section $f$ in the direction $[\dot{h}]$ is
\[ (\nabla_{[\dot{h}]}^\Sot f) (S)= -\tfrac{2\pi i }{16\pi^2}f(S)\int_X 2\Tr (\dot\omega\wedge \Omega),\]
where $\Omega$ is the curvature tensor of $g_{[h]}$ on $X$, and $\omega$ is the connection $1$-form, 
in any orthonormal frame $S$ for $h^0$, parallel transported in the direction of $\pt$ with respect to the metric $G=dt^2+g^t$, where $g^t:=g([h^t])$. Therefore the connection $1$-form $\alpha\in\Lambda^1(\cU)$ of $\nabla^{\mc{L}}$ in the trivialization $f$ is given by 
\[\alpha([\dot{h}])=\tfrac{1}{4\pi i}\int_X \sum_{i,j=1}^3 \dot{\omega}_{ij} \wedge \Omega_{ji}\]
(we note that this does not depend on $S$ anymore). Let $R^G$ be the curvature tensor of $G$, and $R^V$ the curvature of the vertical connection $\nabla^V:=\Pi_{TX}\circ \nabla^G$. As a side note, we remark that this vertical connection is independent on the choice of metric on the horizontal distribution, so we could have chosen in the definition of $G$ any other metric, for instance the one induced from $\Teig$ via the projection. 
We compute 
\begin{align*}
\pt\omega_{ij}(Y)= \lan R^G_{\pt, Y}S_j,S_i\ran=\lan R^V_{\pt, Y}S_j,S_i\ran, && 
\Omega_{ji}(Y_2,Y_3)=\lan R^G_{Y_2,Y_3} S_i,S_j\ran=\lan R^V_{Y_2,Y_3} S_i,S_j\ran
\end{align*}
where the scalar products are with respect to $G$.
This implies 
\[\alpha([\dot{h}])=\tfrac{1}{8\pi i}\int_X\pt\lrcorner \Tr((R^V)^2).\]
The Chern-Simons form of the connection $1$-form $\omega^V$ of $\nabla^V$ in a vertical frame $S$ is a transgression for the Chern-Weil form $\Tr((R^V)^2)$:
\[d\cs(\omega^V)=\Tr((R^V)^2).\]
Writing $d=d^X+d^\rr$ and using Stokes, we get
\[\alpha([\dot{h}])=\tfrac{1}{8\pi i}\left(\int_M\pt\lrcorner \cs(\omega^V)
 \right)+ \tfrac{1}{8\pi i}\pt\left(\int_X \cs(\omega^V)\right)|_{t=0}.
\]
Thus the connection $1$-form of $\nabla^{\mc{L}}$ over $s(\cU)$ satisfies
\[\alpha= \tfrac{1}{8\pi i}\int_{\cMU/s(\cU)} \cs(\omega^V)+\tfrac{1}{8\pi i} d\int_{\cX^\cU/s(\cU)} \cs(\omega^V).\]
The second contribution is an exact form, the curvature of $\nabla^{\mc{L}}$ is therefore the horizontal exterior differential
\[R^\Sot=d\alpha=\tfrac{1}{8\pi i}\int_{\cMU/s(\cU)} d^H\cs(\omega^V).\]
By Stokes, we can add inside the integral the vertical exterior differential, thus
\begin{equation}\label{fpp}
R^\Sot=\tfrac{1}{8\pi i}\int_{\cMU/s(\cU)} d\cs(\omega^V)=\tfrac{1}{8\pi i}\int_{\cMU/s(\cU)} \Tr((R^V)^2)=\tfrac{1}{8\pi i}\int_{\cMU/s(\cU)} \Tr(R^2).
\end{equation}
Here $R$ is the curvature of the vertical tangent bundle of $\cMU\to s(\cU)$ with respect to the natural connection induced by the vertical metric and the horizontal distribution. Notice that the vertical tangent bundle of
the fibration $\cX^\cU\to s(\cU)$ splits orthogonally along $\cMU\to s(\cU)$ into a flat real line bundle corresponding to the normal bundle to $M\subset X$, and the tangent bundle to the fibers of $\cMU$. Thus in the above Chern-Weil integral we can eliminate the normal bundle to $M$ in $X$, which justifies the last equality in \eqref{fpp}.

Next, we compute explicitly this integral along the fibers of the universal curve in terms of the Weil-Petersson form on $\Tei$. Take a $2$-parameters family $h^{t,s}$ in $\Mm$ and let $\dot{H}^t,\dot{H}^s\in \End(TM)$ be defined by
\begin{align*}
\pt h^{t,s}|_{t=s=0}=h(\dot{H}^t\cdot,\cdot),&& \partial_s h^{t,s}|_{t=s=0}=h(\dot{H}^s\cdot,\cdot). 
\end{align*}
where $h:=h^{t,s}|_{t=s=0}$.

Let $X_1,X_2$ be a local frame on $M$, orthogonal at some point $p\in M$ with respect to the metric $h$, and $R$ the curvature of the connection on $TM$ over $\rr^2\times M$. 
\begin{lemma}\label{idecu}
At the point $p\in M$ where the frame $X_j$ is orthonormal we have
\begin{align*}
R_{\partial_s\pt }X_j=&-\tfrac14 [\dot{H}^s,\dot{H}^t]X_j,&& \lan R_{X_1,X_2}X_2,X_1\ran=-1\\
\intertext{and if we choose the family $h$ such that $\dH^t\in V_h$, 
the space of transverse traceless symmetric $2$-tensors, then}
R_{\pt X_j}=&0. 
\end{align*}
\end{lemma}
\begin{proof}
We first compute from Koszul's formula 
\begin{align*}
\lan\nabla_{\pt} X_i,X_j\ran=\demi\pt\lan X_i,X_j\ran=
\demi \lan\dot{H}^t(X_i),X_j\ran,
\end{align*}
so $\na_\pt X_i=\demi \dH^t(X_i)$ and similarly $\na_\ps X_i=\demi \dH^s(X_i)$.
Next we compute
\begin{align*}
\lan R_{\partial_s\pt }X_i,X_j\ran =& \lan\na_\ps \na_\pt X_i,X_j\ran
-\lan\na_\pt\na_\ps X_i,X_j\ran\\
=&\ps \lan \na_\pt X_i,X_j\ran-\lan \na_\pt X_i,\na_\ps X_j\ran
-\pt\lan\na_\ps X_i,X_j\ran +\lan\na_\ps X_i,\na_\pt X_j\ran\\
=&\demi\ps \pt \lan X_i,X_j\ran-\tfrac14 \lan \dH^t(X_i),\dH^s(X_j)\ran
-\demi\pt\ps\lan X_i,X_j\ran +\tfrac14 \lan \dH^s(X_i),\dH^t(X_j)\ran
\end{align*}
which proves the first identity of the lemma. The second identity is simply the fact that metric along the fibers has curvature $-1$. For the third, assume that $\dH^t$ is transverse traceless. At a fixed point $p\in M$
choose a holomorphic coordinate $z=x_1+ix_2$ for $h$ such that $h=|dz^2|+O(|z|^2)$, and choose $X_1=\partial_{x_1},X_2=\partial_{x_2}$. 
Using that $\na_{X_i}X_j=0$ at $p$, we compute at that point
\begin{align*}
\lan\na_\pt\na_{X_1} X_1,X_2\ran=& \pt \lan\na_{X_1} X_1,X_2\ran =\pt(X_1\lan X_1,X_2\ran-\demi X_2\lan X_1,X_1\ran)\\
=&\partial_{x_1}\dH^t_{12}-\demi\partial_{x_2}\dH^t_{11},\\
\lan\na_{X_1} \na_\pt X_1,X_2\ran=& X_1 \lan \na_\pt X_1,X_2\ran= \demi \partial_{x_1}\dH^t_{12}\\
\intertext{which implies at $p$}
\lan R_{\pt, X_1} X_1,X_2\ran=&\demi(\partial_{x_1}\dH^t_{12}-\partial_{x_2}\dH^t_{11}).
\end{align*}
This last quantity vanishes by the Cauchy-Riemann equations when we expand $\dH^t_{ij}$ using $\dH^t=\Re(f(z)dz^2)$ for some holomorphic function $f$. 
\end{proof}
Lemma \eqref{idecu} implies for the trace of the curvature at $p\in M$
\begin{align*}
\Tr(R^2)(\partial_s,\pt,X_1,X_2)=&2\Tr(R_{\ps,\pt}R_{X_1,X_2})=4\lan R_{\ps,\pt}X_1,X_2\ran\lan R_{X_1,X_2} X_2,X_1\ran\\
=&\lan [\dH^s,\dH^t]X_1,X_2\ran = -\lan \dH^s\dH^t JX_2,X_2\ran-\lan \dH^t\dH^s X_1,JX_1\ran\\
=&-\Tr(J\dot{H}^s\dot{H}^t).
\end{align*}

Since $\dH^t$ is transverse traceless, the Weil-Petersson inner product of the vectors $\pt$, $J\partial_s\in T_{h}\Tei$ is just the $L^2$ product
$\int_M\Tr(\dot{H}^t J\dot{H}^s){\rm dvol}_{h}$. The proof is finished by applying \eqref{fpp}.
\end{proof}

The identity \eqref{fpp} expressing the curvature of the Chern-Simons bundle as the fiberwise integral of the Pontrjagin form $\Tr(R^2)$ was proved for arbitrary surface fibrations by U.~Bunke \cite{Bunke} in the context of smooth cohomology.

Since the curvature of the connection $\nabla^{\mc{L}}$ is a $(1,1)$ form, we get the
\begin{cor}
The complex line bundle $\mc{L}$ on $\Tei$ has a holomorphic structure induced by the connection $\nabla^{\mc{L}}$, such that the $\dbar$ operator is the $(0,1)$ component of $\nabla^{\mc{L}}$. 
\end{cor}

\section{Variation of the Chern-Simons invariant and curvature of $\mc{L}$}

In this section we study the covariant derivative of the Chern-Simons $\CS(\theta)$ invariant viewed as a section in the pull-back of Chern-Simons bundle to $T\Teig$.

By Proposition \ref{comparaisons}, Proposition \ref{CSconforme} 
and Lemma \ref{CSisasection}, if $g^t$ is a curve of convex co-compact 
hyperbolic $3$-manifolds, then the invariant $e^{2\pi i\CS^{\Psldc}(g^t,\cdot)}$ 
can be seen as a section of the line bundle $\mc{L}$ over a curve $h_0^t\in \Mm$ induced by the conformal infinities of $g^t$. 
  
\begin{theorem}\label{variationPSL2}
Let $(X,g^t)$, $t\in (-\eps,\eps)$, be a smooth curve of convex co-compact hyperbolic $3$-manifolds 
with conformal infinity a Riemann surface $M$, 
and such that $g$ is isometric near $M$ to the funnel $(0,\eps)_x\x M$
\begin{align}
g^t=\frac{dx^2+h^t(x)}{x^2}, && h(x)=h_0^t+x^2h_2^t+\tfrac{1}{4}x^4h_2^t\circ h_2^t,
\end{align}
with $(h_0^t,h^t_2-\demi h_0^t)\in T\Tei$. Let $S\in C^\infty(\oX,F_0(\oX))$  
be an orthonormal frame for $g^0$ and let $S^t$ be the parallel transport of $S$ in the $t$ direction 
with respect to the metric $G=dt^2+g^t$ on $X\x (-\eps,\eps)$.
Then, setting $\dot{h}_0:=\pl_th_0^t|_{t=0}$ and $h_2:=(h^t_2-\demi h_0^t)|_{t=0}$ so that 
$\dot{h}_0,h_2\in T_{h_0}\Tei$, one has 
\[ \pl_t \CS^{\Psldc}(g^t,S^t)|_{t=0}=
\tfrac{1}{16\pi^2}\int_{M}\Tr(\dot{\hat{\omega}}\wedge\hat{\omega})+
\tfrac{i}{8\pi^2}\cjg ({\rm Id}-iJ)\dot{h}_0, h_2\cjd_{\rm WP}.\]
\end{theorem}

Notice, by Lemma \ref{paralleltr}, that $\hat{S}^t:=x^{-1}S^t$ is parallel for $\hat{G}^t=dt^2+x^2g^t$ and 
thus Theorem \ref{variationPSL2} is sufficient 
to compute the covariant derivative of $e^{2\pi i\CS^{\Psldc}}$ with respect to $\nabla^{\mc{L}}$ in the direction of conformal infinities of hyperbolic metrics on $X$.

Before giving the proof, let us give as an application the variation formula for the renormalized volume.

\begin{cor}\label{0vol}
Let $X^t:=(X,g^t)$ be a smooth curve of convex co-compact hyperbolic $3$-manifolds like in Theorem \ref{variationPSL2}. Then 
\[\pl_t ({\rm Vol}_R(X^t))|_{t=0}=-\tfrac{1}{4}\cjg \dot{h}_0,h_2\cjd_{\rm WP}.\]
\end{cor}
\begin{proof} It suffices to combine Theorem \ref{variationPSL2} with Proposition \ref{comparaisons} and consider 
the imaginary part in the variation formula of $\CS^{\Psldc}$. 
\end{proof}
This formula was proved by Krasnov and Schlenker \cite{KrSc}, using the Schl\"afli formula, in order to show that the renormalized volume is a K\"ahler potential for the Weil-Petersson metric on Teichm\"uller space. The Chern-Simons approach thus provides another proof. 

\begin{proof}[Proof of Theorem \ref{variationPSL2}]
Let $T\in\Lambda^1(X,\End(TX))$ be defined by $T_U(V):=-U\times V$, where $\times$ denotes the vector product with respect to the Riemannian metric. Clearly $T$ is anti-symmetric.
We consider a $1$-parameter family of metrics on $X$ hyperbolic outside a compact set, $g^t=\frac{dx^2+h^t(x)}{x^2}$, and we define a Riemannian metric on $\rr\times X$ by
\[G=dt^2+g^t.\]
Recall that for every fixed $t$, the metrics $h^t(x)$ and $h^t_0:=h^t(0)$ on $M$ are related by \eqref{metrhyp}.

For a given section $S$ in the orthonormal frame bundle for $g_0$, we define $S^t$ as the parallel transport in the $t$ direction of $S$, more precisely $\na_\pt S^t_j=0$ for $j=1,2,3$. Here and in what follows, $\na$ denotes the Riemannian connection for $G$. Since the integral curves of $\pt$ are geodesics, it follows that $S^t$ is an orthonormal frame for $g^t$.

Consider the connections $D^t = \na^{g^t}+iT^t$ on $T_\cc X$ corresponding to the metric $g^t$. In the trivialization given by the section $S^t$, the connection form is $\theta^t=\omega^t+iT^t$. It is a $\sot\otimes \cc$-valued $1$-form with real and imaginary parts 
\begin{align*}
\omega^t_{ij}(Y)=\lan\na^{g^t}_Y S^t_j,S^t_i\ran_{g^t},&& T^t_{ij}(Y)=\lan Y\times^t S_i^t,S_j^t\ran_{g^t}.
\end{align*}

We first compute the variation with respect to $t$ of the Chern-Simons form of $\theta^t$ on $X$. 
In what follows, we will drop the upperscript $t$ when we evaluate at $t=0$ and we shall use a dot
to denote the $t$-derivative at $t=0$.
Substituting in \eqref{varcs} for $\theta=\omega+iT=\hat{\omega}-\alpha+iT$ like in \eqref{omegavs'}, with $\hat{\omega}$ the connection form of the Levi-Civita connection of the conformally compactified metric $\hat{g}=x^2g$, we get
\begin{equation}\label{varcspsld} \begin{split}
\pt \cs(\theta^t)|_{t=0}=&d(\Tr(\dot\omega\wedge\omega)-\Tr(\dot{T}\wedge T)
+i(\Tr(\dot\omega\wedge T)+\Tr(\dot{T}\wedge\omega)))+2\Tr(\dtheta\wedge\Omega^\theta)\\
=&d\Tr(\dot{\hat{\omega}}\wedge \hat{\omega})+d[\Tr(\dot{\alpha}\wedge \hat{\omega})+ \Tr(\dot{\hat{\omega}}\wedge\alpha)]+id[\Tr(\dot\omega\wedge T) +\Tr(\dot{T}\wedge\omega)]\\
&+d[\Tr(\dot{\alpha}\wedge \alpha)-\Tr(\dot{T}\wedge T)]+2\Tr(\dtheta\wedge\Omega^\theta).
\end{split}
\end{equation}
Observe that if $g^t$ is a variation through hyperbolic metrics on $X$ then $\Omega^\theta$ vanishes. We claim that
in the variation formula for the Chern-Simons invariant of $\theta^t$, the finite parts corresponding to the terms $\Tr(\dot{\alpha}\wedge \alpha)$ and $\Tr(\dot{T}\wedge T)$ vanish. 
We start with the term $\Tr(T\wedge \dot{T})$:
\begin{lemma}
We have  ${\rm FP}_{\eps=0}\int_{x=\eps}\Tr(T\wedge \dot{T})=0$.   
\end{lemma}
\begin{proof} Let $Y_1,Y_2$ be vector fields on $M$, independent of $t$, then using that $\nabla^G_{\pl_t}S_j=0$, we have $\dot{T}_{ij}(Y_k)=\lan\nabla^G_\pt Y_k\times S_i,S_j\ran$ so
\begin{align*}
\Tr(T\wedge \dot{T})(Y_1,Y_2)=&\sum_{i,j} \lan Y_1\times S_i,S_j\ran\lan \nabla^G_\pt Y_2\times S_i,S_j\ran
-\lan Y_2\times S_i,S_j\ran\lan \nabla^G_\pt Y_1\times S_i,S_j\ran\\
=&\lan Y_1,\nabla_\pt Y_2\ran-\lan Y_2,\nabla_\pt Y_1\ran
\end{align*}
which is zero because by Koszul, $\lan Y_1,\nabla_\pt Y_2\ran=\tfrac12 (L_\pt G)(Y_1,Y_2)$ is symmetric in $Y_1,Y_2$.
\end{proof}

\begin{lemma}\label{alphadot}
For $\eps>0$ sufficiently small we have
\[\Tr(\dot{\alpha}\wedge \alpha)|_{x=\eps}=0\]
\end{lemma}
\begin{proof}
Let $Y_1,Y_2$ be tangent vector fields to $M$, independent of $t$. Notice that for $S_i$ parallel with respect to $\nabla^G_{\pl_t}$ 
then $\hat{S}_i=x^{-1}S_i$ is parallel with respect to $\nabla^{\hat{G}}_{\pl_t}$ where $\hat{G}=dt^2+\hat{g}^t$. Then since $\pl_x$ is also killed by 
$\nabla^{\hat{G}}_{\pl_t}$
\[\begin{split}
x\Tr(\dot{\alpha}\wedge\alpha)(Y_1,Y_2)=&\pl_t[\hat{g}^t(Y_1,\hat{S}^t_i)\hat{S}^t_j(x)-\hat{g}^t(Y_1,\hat{S}^t_j)\hat{S}^t_i(x)]|_{t=0}\big(\hat{g}(Y_2,\hat{S}_j)\hat{S}_i(x)-\hat{g}(Y_2,\hat{S}_i)\hat{S}_j(x)\big)\\
& -{\rm Sym}(Y_1\to Y_2)\\
=&-2\hat{G}(\nabla^{\hat{G}}_{\pl_t}Y_1,Y_2)
+2\hat{G}(\nabla^{\hat{G}}_{\pl_t}Y_2,Y_1)=-2\hat{G}(\nabla^{\hat{G}}_{Y_1}\pl_t,Y_2)+2\hat{G}(\nabla^{\hat{G}}_{Y_2}\pl_t,Y_1)\\
x\Tr(\dot{\alpha}\wedge\alpha)(Y_1,Y_2)=& 2\hat{G}(\nabla^{\hat{G}}_{Y_1}Y_2,\pl_t)-2\hat{G}(\nabla^{\hat{G}}_{Y_2}Y_1,\pl_t)=2\hat{G}([Y_1,Y_2],\pl_t)=0
\end{split}\]
and this finishes the proof.
\end{proof}

We now consider the term $\Tr(\dot{\hat{\omega}}\wedge \alpha)+\Tr(\dot{\alpha}\wedge \hat{\omega})$.
\begin{lemma}
Let $\dot{H}_0$ and $A$ be the symmetric endomorphism on $TM$ defined by $\dot{h}_0(\cdot,\cdot)=h_0(\dot{H}_0\cdot,\cdot)$ and $h_2(\cdot,\cdot)=h_0(A\cdot,\cdot)$. We have the following identity
\[{\rm FP}_{\eps\to 0}\left(\Tr(\dot{\hat{\omega}}\wedge \alpha)+\Tr(\dot{\alpha}\wedge \hat{\omega})\right)|_{x=\eps}=
2 \int_{M}\Tr(J\dot{H}_0A)\vol_{h_0}
\]
where $J$ is the complex structure on $TM$.
\end{lemma}
\begin{proof}
First, from the proof of Proposition \ref{CSconforme}, we know that ${\rm FP}_{\eps\to 0}\Tr(\alpha\wedge \hat{\omega})|_{x=\eps}=0$, and therefore  
\[{\rm FP}_{\eps\to 0}\left(\Tr(\dot{\hat{\omega}}\wedge \alpha)+\Tr(\dot{\alpha}\wedge \hat{\omega})\right)|_{x=\eps}=
2\,{\rm FP}_{\eps\to 0}\Tr(\dot{\hat{\omega}}\wedge \alpha).\]
Now, for $Y_1,Y_2$ tangent to $M$ and independent of $t$, we can use that $\nabla^{\hat{G}}_{\pl_t}\hat{S}_i=0$ and 
$\hat{\omega}_{ij}(Y)=\hat{g}(\nabla^{\hat{g}}_{Y}\hat{S}_j,\hat{S}_i)=\hat{G}(\nabla^{\hat{G}}_{Y}\hat{S}_j,\hat{S}_i)$ to deduce  
\[\dot{\homega}_{ij}(Y)=\pt\lan \nabla_Y  \hat{S}_j,\hat{S}_i\ran=\lan \nabla_\pt\nabla_Y  \hat{S}_j,\hat{S}_i\ran
=\lan R_{\pt Y} \hat{S}_j,\hat{S}_i\ran\]
where $\hat{R}$ is the curvature tensor of $\hat{G}$, therefore
\begin{align*}
\Tr(\dot{\hat{\omega}}\wedge \alpha)(Y_1,Y_2)=&\sum_{i,j} \lan R_{\pt Y_1} \hat{S}_j,\hat{S}_i\ran
(\lan Y_2,S_j\ran S_i(a)-\lan Y_2,S_i\ran S_j(a))\\
&-\lan \hat{R}_{\pt Y_2} \hat{S}_j,\hat{S}_i\ran
(\lan Y_1,S_j\ran S_i(a)-\lan Y_1,S_i\ran S_j(a))\\
=&2(\lan \hat{R}_{\pt Y_2} Y_1, x^{-1}\px\ran-\lan \hat{R}_{\pt Y_1} Y_2, x^{-1}\px\ran\\
=&2x^{-1} \lan \hat{R}_{\pl_t, \pl_x}Y_1,Y_2\ran
\end{align*}
by Bianchi. Since we are interested in the finite part, we can modify $Y_1,Y_2$ by a term of order $x^2$ without changing the result, and we will take $\til{Y}^t_i=(1-\demi x^2A^t)Y_i$ where the endomorphism $A^t$ of $TM$ is defined by $h^t_2(\cdot,\cdot )=h^t_0(A^t\cdot,\cdot)$.  Then 
\[\begin{split}\label{hatGR'}
\hat{G}(\hat{R}_{\pl_t, \pl_x}Y_1,Y_2)=&-\pl_x \left(\hat{g}(\nabla^{\hat{G}}_{\pl_t}\til{Y}^t_1,\til{Y}^t_2)\right)|_{t=0}+O(x^2)\\
=& -\demi \pl_x\left( \pl_t(\hat{g}^t(\til{Y}^t_1,\til{Y}^t_2))+\hat{g}^t([\pl_t,\til{Y}^t_1],\til{Y}^t_2)|_{t=0}-\hat{g}^t([\pl_t,\til{Y}^t_2],\til{Y}_1)|_{t=0}\right)+O(x^2).
\end{split}\]
The term  $\pl_t(\hat{g}^t(\til{Y}^t_1,\til{Y}^t_2))|_{t=0}$ is easily seen to be a $\dot{h}_0(Y_1,Y_2)+O(x^3)$ 
by using that $\hat{g}^t=dx^2+h^t_0+x^2h^t_0(A^t\cdot,\cdot)+O(x^4)$, while 
the other two terms are 
\[\begin{split}
\hat{g}^t([\pl_t,\til{Y}^t_1],\til{Y}^t_2)|_{t=0}-\hat{g}^t([\pl_t,\til{Y}^t_2],\til{Y}^t_1)|_{t=0}
=&\demi x^2h_0(\dot{A}Y_1,Y_2)-\demi x^2h_0(\dot{A}Y_2,Y_1)+O(x^4)\\
=& \demi x^2h_0((\dot{A}-\dot{A}^T)Y_1,Y_2)+O(x^4).
\end{split}
\]
but since $A^t$ is symmetric with respect to $h_0^t$, we deduce by differentiating at $t=0$ that 
$\dot{A}-\dot{A}^T=(\dot{H}_0A)^T-\dot{H}_0A$ and therefore
\[\begin{split}
\hat{g}^t([\pl_t,\til{Y}^t_1],\til{Y}^t_2)|_{t=0}-\hat{g}^t([\pl_t,\til{Y}^t_2],\til{Y}^t_1)|_{t=0}
=&\demi x^2h_0(\dot{H}_0AY_1,JY_1)+\demi x^2h_0(\dot{H}_0AY_2,JY_2)+O(x^4)\\
=& -\demi x^2 \Tr(J\dot{H}_0A)+O(x^4).
\end{split}\]
We conclude that the limit of $\frac{2}{x}\hat{G}(\hat{R}_{\pl_t, \pl_x}Y_1,Y_2)$ as $x\to 0$ 
is given by $\Tr(J\dot{H}_0A)$.
\end{proof}

Next, we reduce the sum $\Tr(\dot{T}\wedge \omega)+\Tr(\dot{\omega}\wedge T)$ as follows:
\begin{lemma}\label{reduction}
We have the following identity 
\[{\rm FP}_{\eps=0}\int_{x=\eps}\Tr(\dot{T}\wedge \omega)+\Tr(\dot{\omega}\wedge T)
=2{\rm FP}_{\eps=0}\int_{x=\eps}\Tr(\dot{\omega}\wedge T).\]
\end{lemma}
\begin{proof}
It suffices use \eqref{TromegaT} to deduce that $\pl_t{\rm FP}_{\eps=0}\int_{x=\eps} \Tr(\omega\wedge T)=8\pi \pl_t(\chi(M))=0$.
\end{proof}

\begin{prop}\label{Trdotomega}
Let $\dot{H}_0$ be the endomorphism on $TM$ defined by $\dot{h}_0(\cdot,\cdot)=h_0(\dot{H}_0\cdot,\cdot)$. Then near $x=0$ we have 
\[\Tr(\dot{\omega}\wedge T)=[-x^{-2}\Tr(\dot{H}_0)+\Tr(\dot{A})-\tfrac12\Tr(A)\Tr(\dot{H}_0)+\Tr(\dot{H}_0A)]\vol_{h_0}
+O(x^2).\]
\end{prop}
\begin{proof}
Notice that for every $Y$ tangent to $X$ we have $\omega^{t}_{ij}(Y)=\omega_{ij}(Y)$, as a simple consequence of the Koszul formula.
For a vector field $Y$ on $X$ extended on $\rr\times X$ to be constant with respect to the flow of $\pt$ we compute
\[(\partial_t\omega_{ij})(Y)|_{t=0}=\pt \lan\na_Y S^t_j,S^t_i\ran|_{t=0}=\lan\na_\pt\na_Y S^t_j,S^t_i\ran|_{t=0}=\lan R_{\pt, Y}S_j,S_i\ran\]
where $\lan \cdot,\cdot \ran$ denotes the metric $G$.
In the last equality we have used the fact that $S^t_i$ is parallel in the direction of $\pt$ and the vanishing of the bracket
$[\pt,Y]$. By the symmetry of the Riemannian curvature tensor, we rewrite the last term as $-\lan R_{S_i,S_j}\pt,Y\ran$. It follows that 
\begin{equation}\label{tdotc}\begin{split}
\Tr(\dot{\omega}\wedge T)(Y_1,Y_2)=&\sum_{i,j=1}^3 -\lan R_{S_iS_j}\pt, Y_1\ran \lan Y_2\times S_j,S_i\ran+
\lan R_{S_iS_j}\pt, Y_2\ran \lan Y_1\times S_j,S_i\ran\\
=&\sum_{j=1}^3 -\lan R_{Y_2\times S_j, S_j}\pt,Y_1\ran+ \lan R_{Y_1\times S_j, S_j}\pt,Y_2\ran\\
=&E(Y_1,Y_2)-E(Y_2,Y_1)
\end{split}\end{equation}
where we have defined
\[E(Y,Z):=\sum_{j=1}^3 \lan R^G_{Y\times S_j, S_j}\pt,Z\ran.\]
For every vector field $Y$ on $M$, define a vector field $\tY^t$ on a neighborhood of $M$ in $X$ by
\[\tY^t=(1+\tfrac{x^2}{2}A^t)^{-1}Y\]
where $h_2^t=h_0^t(A^t\cdot,\cdot)$.
From \eqref{metrhyp} we see that for any orthonormal frame $Y_1,Y_2$ on $M$ for $h_0$, 
the frame $\tY^t_1,\tY^t_2$ at $t=0$ is also orthonormal on $X$. 
The complex structure $J$ on $\{t\}\times\{x\}\times M$ satisfies
$J\tY^t=x\px\times^t \tY^t$, so in particular $J\tY=\widetilde{JY}$ at $t=0$.
\begin{lemma}\label{leme}
Let $Y,Z$ be vector fields on $M$. Then near $x=0$ we have the expansion
\begin{equation}\label{technre}\begin{split}
E(J\tY,\tZ)=&
x^{-2}\dot{h}_0(Y,Z)-\tfrac{1}{2}(h_0(\dot{A}Y,Z)+h_0(Y,\dot{A}Z))
-\dot{h}_0(AY,Z)+O(x^2).
\end{split}\end{equation}
\end{lemma}
\begin{proof}
The expression defining $E$ is independent of the orthonormal frame $S_j$ for $g$, 
thus we can compute it using the frame $x\tY,xJ\tY,x\px$ (all these are at $t=0$):
\begin{align*}
E(J\tY,\tZ)=& 2 \lan R_{\tY,x\px}\pt,\tZ\ran.
\end{align*}
Note the following identities:
\begin{align}\label{ide}
\tY^t=Y-\tfrac{x^2}{2}A^tY+O(x^4),&&
[x\px,x\tY^t]=x\tY^t-x^3A^tY+O(x^5),&&\na^G_{x\px} x\tZ^t=O(x^4).
\end{align}
Also, note that $\na^{G}_{\px} \pt=0$. Using these facts, we get
\begin{equation}\label{expna}\begin{split}
\lan\na^G_{x\tY}\pt,x\tZ\ran=&\tfrac12(L_{\pt}G)(x\tY^t,x\tZ^t)|_{t=0}\\ 
=&\tfrac12( (\pt h^t_0((1+\tfrac{x^2}{2}A^t)\cdot,(1+\tfrac{x^2}{2}A^t)\cdot))(\tY,\tZ))|_{t=0}+O(x^4)\\
=&\tfrac12\dot{h}_0(Y,Z)+\tfrac{x^2}{4}(h_0(\dot{A}Y,Z)+h_0(Y,\dot{A}Z))+O(x^4),
\end{split} \end{equation}
therefore by using \eqref{ide}
\begin{align*}
\lan R_{x\px,x\tY}\pt,x\tZ\ran=&x\px \lan\na^G_{x\tY}\pt,x\tZ\ran
-\lan\na^G_{[x\px,x\tY]}\pt,x\tZ\ran\\
=&(x\px-1) \lan\na^G_{x\tY}\pt,x\tZ\ran +x^2\lan\na^G_{x\widetilde{AY}}\pt,x\tZ\ran\\
=&-\tfrac12\dot{h}_0(Y,Z)+\tfrac{x^2}{4}(h_0(\dot{A}Y,Z)+h_0(Y,\dot{A}Z))
+\tfrac{x^2}{2}\dot{h}_0(AY,Z)
+O(x^4)
\end{align*}
(in the last step we have used \eqref{expna} for $AY$ in the place of $Y$).
Using the tensoriality of the curvature to get out the factors of $x$, we proved the lemma.
\end{proof}
Let us now write (all what follows is at $t=0$) 
\[Y=\tY+\frac{x^2}{2}AY+O(x^4)=\tY+\frac{x^2}{2}\widetilde{AY}+O(x^4).\]
By linearity we get
\begin{align}\label{expre}
E(Y,Z)=&E(\tY,\tZ)+\frac{x^2}{2}(E(\widetilde{AY},Z)+E(Y,\widetilde{AZ}))+O(x^2).
\end{align}
Assume now that $Y_j$ have been chosen at a given point on $M$ as (orthonormal) eigenvectors of $A$ for $h_0$
of eigenvalue $\lambda_j$, with $JY_1=Y_2$. Then from \eqref{expre}  we get 
\begin{align*}
E(Y_2,Y_1)= & (1+\tfrac{x^2}{2}(\lambda_1+\lambda_2))E(\tY_2,\tY_1),\\
E(Y_1,Y_2)= & (1+\tfrac{x^2}{2}(\lambda_1+\lambda_2))E(\tY_1,\tY_2),\\
\intertext{therefore from \eqref{tdotc} and Lemma \ref{leme}}
\Tr(\dot{\omega}\wedge T)(Y_1,Y_2)= & E(Y_1,Y_2)-E(Y_2,Y_1)\\
=&(1+\tfrac{x^2}{2}\Tr(A))(E(\tY_1,\tY_2)-E(\tY_2,\tY_1))\\
=&-(1+\tfrac{x^2}{2}\Tr(A))(E(J\tY_1,\tY_1)+E(J\tY_2,\tY_2))\\
=&(1+\tfrac{x^2}{2}\Tr(A))(-x^{-2}\Tr(\dot{H^0}) +\Tr(\dot{A})+\Tr(\dot{H}_0A)+O(x^2)\\
\Tr(\dot{\omega}\wedge T)(Y_1,Y_2)=&-x^{-2}\Tr(\dot{H}_0)+\Tr(\dot{A})-\tfrac12 \Tr(A)\Tr(\dot{H}_0)+\Tr(\dot{H}_0A)+O(x^2).
\end{align*}
which is the claim of Proposition \ref{Trdotomega}.
\end{proof}

We are now in position to finish the proof of Theorem \ref{variationPSL2}. Since we consider a family of hyperbolic metrics $g^t$, we have $\Tr(A^t)=-\demi \scal_{h^t_0}$ by \eqref{condA} so by Gauss-Bonnet the following integral is constant in $t$:
\begin{equation}\label{gb}
\int_M \Tr(A^t)\vol_{h^t_0}=2\pi\chi(M).
\end{equation} 
Using $\pt \vol_{h_0^t}|_{t=0}=\demi\Tr(\dot{h}_0)\vol_{h_0}$ we deduce by differentiating \eqref{gb}
that 
\[\int_M (\Tr(\dot{A})+\tfrac12 \Tr(A)\Tr(\dot{H}_0))\vol_{h_0}=0\]
so \[\int_{x=\epsilon} \Tr(\dot{\omega}\wedge T)=-\epsilon^{-2} \pt \Vol(M,h_0)+\int_M (2\Tr(\dot{A})
+\Tr(\dot{H}_0A))\vol_{h_0}+O(\epsilon^2).\]
This achieves the proof of Theorem \ref{variationPSL2}.
\end{proof}

\section{Chern-Simons line bundle and determinant line bundle} \label{scsd}

Ramadas-Singer-Weitsman \cite{RSW} introduced the Chern-Simons line bundle  
on the moduli space $\mc{A}_F^s/\mc{G}$ of irreducible flat ${\rm SU}(2)$ connections up to gauge, 
they showed that it has a natural connection whose curvature is (up to a factor of $i$)
the standard symplectic form, and a natural Hermitian structure. 
Quillen \cite{Qu} defined the determinant line bundle over the 
space $\{\bar{\pl}_A; A\in \mc{A}_F^s\}$ of d-bar operators for a given 
complex structure on the surface $M$: he showed that it descends to $\mc{A}_F^s/\mc{G}$ as a Hermitian line bundle with a natural connection and with curvature the standard symplectic form (up to a factor of $i$). Ramadas-Singer-Weitsman proved that these bundles are isomorphic as Hermitian line bundle with connection 
over $\mc{A}_F^s/\mc{G}$. Moreover their curvature form is of $(1,1)$ type with respect 
to the natural complex structure on $\mc{A}_F^s/\mc{G}$ and therefore the line bundle admits a holomorphic structure.
In what follows, we shall construct, in particular cases, a similar isomorphism using our Chern-Simons invariant and the determinant of the Laplacian.

\subsection{The submanifold $\mc{H}$ of hyperbolic $3$-manifolds}

Let us be more precise and first make the following assumption: if $\Tei$ is Teichm\"uller space
for a given oriented surface $M$ of genus $\bg$ (possibly not connected) and $h_0\in\Teig$, we assume that we fix a convex co-compact $3$-manifold $X$ with conformal boundary $(M,h_0)$.
\begin{prop}\label{assumption}
There exists a neighborhood $\cU\subset\Teig$ of $h_0$ and a smooth map $F:\cU\to\cun(\oX,S^2_+({^0T}^*\oX)))$
such that $F(h)$ is hyperbolic convex-cocompact with conformal boundary $(M,h)$ for all $h\in\cU$.
\end{prop} 
\begin{proof} The proof is written for instance in \cite{MoSc}. A quasiconformal approach can be found for instance 
in Marden \cite{Marden}.
\end{proof}

This map induces by Lemma \ref{identiffunhyp} a local section in the tangent bundle of $\Teig$.
By Mostow rigidity \cite[Theorem 2.12]{Marden} and Marden \cite[Theorem 3.1]{Marden}, this section is unique and extends to a global smooth section $\sigma:\Teig\to T\Teig$. The graph 
\begin{equation}\label{mcH}
\mc{H}:=\{(h,\sigma(h))\in T\Tei; h \in \Tei\}
\end{equation}
is then a smooth submanifold of $T\Tei$ of dimension $\dim \Tei$.
By uniqueness, the subgroup of modular transformations of $M$ consisting of classes of diffeomorphisms which extend to $\oX$ leaves this section invariant, therefore $\sigma$ descends to any quotient of $\Teig$ by such a subgroup.
For instance, this applies to the deformation space of a given convex co-compact hyperbolic $3$ manifold $X=\Gamma\backslash \hh^3$, which is the quotient $\Tei_X:=\Tei/{\rm Mod}_X$ of $\Tei$ by the subgroup ${\rm Mod}_X$ defined 
in \S 4.

Let us introduce a new connection on the pull-back of $\cL_\bg$ to $T\Teig$, for which the $\Psldc$ Chern-Simons section is flat along the deformation space of hyperbolic metrics on $X$. 
\begin{definition}
We define the connection $\nabla^{\mu}$ on $\mc{L}$ by 
\[ \nabla^{\mu} = \nabla^{\mc{L}}+\tfrac{1}{2\pi}\Phi^*\mu^{1,0}\]
where 
$\mu$ is the Liouville $1$-form on $T^*\Tei$ and $\Phi:T\Tei\to T^*\Tei$ is the isomorphism induced by the Weil-Petersson metric. In what follows we will omit the identification $\Phi$.
\end{definition}
This connection is \emph{not} Hermitian with respect to $\cjg \cdot,\cdot\cjd_{\CS}$ since the form $\mu^{1,0}$ is not purely imaginary. 
The Chern-Simons line bundle $\mc{L}$ equipped with the 
connection $\nabla^{\mu}$ has curvature 
\begin{equation}\label{curvaturepsl2}
\Omega_{\nabla^{\mu}}=\tfrac{i}{8\pi}\Omega_\WP+ \tfrac{1}{2\pi}\dbar \mu^{1,0}
\end{equation}
with real part ${\rm Re}(\Omega_{\nabla^\mu})=\tfrac{1}{4\pi} d\mu$, where here and below, $\omega_{\rm WP}$ is understood as 
$\pi^*\omega_{\rm WP}$ if $\pi=T\Tei\to \Tei$ is the projection on the basis.

\begin{theorem}\label{lagrangian}
The Chern-Simons invariant $e^{2\pi i \CSP}$ restricted to the submanifold $\mc{H}$ in \eqref{mcH} is a parallel section of $\mc{L}|_{\mc{H}}$ for the connection $\nabla^{\mu}$. As a consequence, $\mc{H}$ is a Lagrangian submanifold of $T\Tei$ for the standard symplectic Liouville form $d\mu$ on $T\Tei$ obtained from pull-back by the duality isomorphism $T\Tei\to T^*\Tei$ induced by $\cjg\cdot,\cdot\cjd_{\rm WP}$.
\end{theorem}
\begin{proof} This is a direct consequence of the variation formula in Theorem \ref{variationPSL2} and the definition of the connection $\nabla^{\mu}$.
\end{proof}

It is proved by Krasnov \cite{Kra} for Schottky cases 
and more generally by Takhtajan-Teo \cite{TakTeo} for Kleinian groups of class A (see also Krasnov-Schlenker \cite{KrSc} for quasi-Fuchsian cases), that 
\[\bbar{\pl}\pl ({\rm Vol}_R)= \tfrac{i}{16}\omega_{\rm WP}\]
using previous work of Takhtajan-Zograf \cite{TakZo} on Liouville functional. Here
we use our convention for Weil-Petersson metric. The Theorem above generalizes these results providing a unified treatment:
\begin{cor}\label{kahlerpot}
For $h_0$ in an open set $U\subset \Tei$, let $X_{h}=(X,g_{h})$ be a smooth family of convex co-compact hyperbolic $3$-manifolds with conformal infinity parametrized smoothly by $h\in U$, then 
\[ \bbar{\pl}\pl {\rm Vol}_R(X_{h}) =\tfrac{i}{16}\omega_{\WP}.\]  
\end{cor}
\begin{proof} Let $\sigma: U\subset \Tei \to T\Tei$ be the section $h\to (h,\sigma(h))$ parametrizing the submanifold $\mc{H}$. We consider ${\rm Vol}_R(X_{h})$ as a function on $U$. By Corollary \ref{0vol}, we have for $\dot{h}\in T_{h}\Tei$
\[\pl {\rm Vol}_R(X_{h}).\dot{h}=-\tfrac{1}{4}\sigma^*\mu^{1,0}(\dot{h}).\]  
From the vanishing of the curvature $\Omega_{\mc{L}}$ on $\mc{H}$ and the formula \eqref{curvaturepsl2}, we obtain for any $\dot{h},\dot{\ell}\in T_{h}\Tei$
\[ d\mu^{1,0}_{\sigma(h)}(d\sigma.\dot{h},d\sigma.\dot{\ell}) =
-\tfrac{i}{4}\omega_{\WP}(\dot{h},\dot{\ell})\]
and since $\sigma^{*}d\mu^{1,0}_{h}(\dot{h},\dot{\ell})=d(\sigma^{*}\mu^{1,0})(\dot{h},\dot{\ell})
=\bar{\pl}\pl {\rm Vol}_R(X_{h})(\dot{h},\dot{\ell})$, the proof is finished. 
\end{proof}

\subsection{An isomorphism with the determinant line bundle}\label{detline}
Finally, we construct an explicit isomorphism of Hermitian line bundles between $\mc{L}_{\mc{H}}$ and the 
determinant line bundle in the particular cases of quasi-Fuchsian and Schottky manifolds.
 
Let $M$ be a marked Riemann surface of genus $\bg$, i.e., a surface with a distinguished set of generators $\alpha_1,\dots,\alpha_\bg,\beta_1,\dots,\beta_\bg$
of $\pi_1(M,x_0)$ for some $x_0\in M$. With respect to this marking, if a complex structure is given on $M$, there is a basis $\varphi_1,\dots,\varphi_\bg$ of holomorphic $1$-forms such that $\int_{\alpha_j}\varphi_i=\delta_{ij}$ and this defines the period matrix $(\tau_{ij})=(\int_{\beta_j}\varphi_i)$ whose imaginary part is positive definite since 
$2{\rm Im}\, \tau_{ij}=\cjg \varphi_i,\varphi_j\cjd$.  Schottky groups are free groups generated by $L_1,\dots,L_\bg\in \Psldc$  which map circles $C_1\dots , C_\bg\subset \hat{\cc}=\cc\cup\{\infty\}$ to other circles 
$C_{-1},\dots,C_{-\bg}\in\hat{\cc}$ (with orientation reversed). Each element $\gamma\in \Gamma$ is conjugated in $\Psldc$ to $z\to q_\gamma z$ for some $q_\gamma\in \cc$
with $|q_\gamma|<1$, called the multiplier of $\gamma$. 
The quotient of the discontinuity set $\Omega_\Gamma$ of $\Gamma$ by $\Gamma$ is a closed Riemann surface and
every closed Riemann surface of genus $\bg$ can be represented in this manner by a result of Koebe, see \cite{Fo}. 
The Schottky group is marked if each $C_k$ is homotopic to 
$\alpha_k$ in the quotient $\Gamma\backslash\Omega_{\Gamma}$. 
The marked group is unique up to a global conjugation in $\Psldc$ and a normalization condition 
(by assigning the $2$ fixed points of $L_1$ and one of $L_2$) can be set to fix it. One then obtains the Schottky space 
$\mathfrak{S}$ which covers the moduli space (i.e., the set of isomorphism classes of compact Riemann surface of 
genus $\bg$) but is covered by Teichm\"uller space $\Tei$ whose points are isomorphism classes of marked compact Riemann surfaces. 

Since any Schottky group $\Gamma\subset \Psldc$ acts as isometries on $\hh^3$ as a convex co-compact group, there is a canonical hyperbolic $3$-manifold $\Gamma\backslash \hh^3$ with conformal infinity given by $\Gamma\backslash \Omega_\Gamma$. This manifold denoted $X$ is a handlebody with conformal boundary $M$.
Let $\Diff_X(M)$ the group of diffeomorphisms of $M$ which extend to $\oX$ factored by the group $\mc{D}_0$ of diffeomorphisms of $M$ homotopic to ${\rm Id}$.  

The Chern-Simons line bundle defined on $\Tei$ above is acted  upon by $\Diff_X(M)$, thus it descends 
to the Schottky space $\mathfrak{S}$ which is a quotient of $\Tei$ by a subgroup of $\Diff_X(M)$, we denote it $\mc{L}_{\mathfrak{S}}$.
The connection on $\mc{L}$ over $\Tei$ defined in Subsection \ref{connectionsL} is $\Diff_X(M)$ invariant, hence it descends to $\mathfrak{S}$. The Liouville form on $T\Tei$ is  $\Diff(M)$ invariant and thus also descends to $T\mathfrak{S}$, then the connection $\nabla^{\mc{L}}$ descends to $T\mathfrak{S}$, we denote it 
$\nabla^{\mathfrak{S}}$. Again, we can define the Lagrangian submanifold $\mc{H}\subset T\mathfrak{S}$ consisting of those funnels which extend 
to Schottky $3$-manifolds. 
The operator $\pl_{\Gamma}: 
\cun(M)\to \cun(M,{\Lambda}^{1,0}M)$ for a given complex structure induced by $\Gamma$ on $M$ 
is Fredholm on Sobolev spaces and, considered as a family of operators parametrized by points $\Gamma\in \mathfrak{S}_\bg$, one can define its determinant line bundle $\det(\pl)$ 
of $\pl$, as in Quillen \cite{Qu}, to be at $\Gamma$ the line\footnote{We have ignored the kernel of $\pl$ since it is only made of constants with norm given essentially by the Euler characteristic of $M$ by Gauss-Bonnet, therefore not depending at all on the complex structure on $M$.}  
\[ {\rm det}(\pl_{\Gamma}):= \Lambda^{\bg}({\rm coker}\, \pl_\Gamma) \]
when $\bg\in \nn$ and ${\rm coker}\, \pl_\Gamma=\ker (\bbar{\pl}_{\Gamma} : \cun(M,{\Lambda}^{1,0})\to \cun(M,\Lambda^2(M)))=:H^{0,1}(\Gamma\backslash \Omega_\Gamma)$ is the vector space of holomorphic $1$-forms 
on $M\simeq \Gamma\backslash \Omega_\Gamma$. The line bundle $\det(\pl)$ over $\mathfrak{S}$ is a holomorphic line bundle with a holomorphic canonical section 
\begin{equation}\label{cansection}
\varphi := \varphi_1\wedge \dots \wedge \varphi_{\bg}
\end{equation}
and is equipped with a Hermitian norm, called Quillen metric, defined as follows: for each Riemann surface $\Gamma\backslash \Omega_\Gamma $ with $\Gamma \in \mathfrak{S}_\bg$, let 
$h_0$ be the associated hyperbolic metric obtained by uniformisation and 
define $\det' \Delta_{h_0}$ the determinant of its Laplacian, as defined in Ray-Singer \cite{RaySin}, then 
the Hermitian metric on $\det(\pl)$ is given at $\Gamma\in \mathfrak{S}$ by  
\begin{equation}\label{quillenmetric}
\|\varphi\|_{Q}^2:= \frac{\|\varphi\|^2_{h_0}}{\det' \Delta_{h_0}} = \frac{\det {\rm Im}\,\tau}{\det' \Delta_{h_0}}
\end{equation}
where $\|\cdot\|_{h_0}$ is the Hermitian product on $\Lambda^{\bg}({\rm coker}\, \pl_\Gamma)$ induced by the 
metric $h_0$ on differential forms on $M$. We denote by $\nabla^{\det}$ the unique Hermitian connection 
associated to the holomorphic structure on $\det \pl$ and the Hermitian norm $\|\varphi\|_{Q}$.

To state the isomorphism between powers of Chern-Simons line bundle and a power of the determinant line bundle, we will use a formula proved by Zograf \cite{Zo1,Zo2} and generalized by McIntyre-Takhtajan \cite{McTa}
\begin{theorem} {\bf{[Zograf]}}\label{thzograf}
There exists a holomorphic function $F(\Gamma):\mathfrak{S}_\bg \to
\mathbb{C}$ such that
\begin{equation}\label{e:Zograf}
\frac{\mathrm{det}'\Delta_{h_0}}{\mathrm{det}\, \mathrm{Im}\, \tau} = c_\bg
\exp\left( \frac{{\rm Vol}_R(X)}{3\pi} \right) |F(\Gamma)|^2
\end{equation}
where $c_\bg$ is a constant depending only on $\bg$  where 
 $X=\Gamma\backslash\hh^3$ when we see $\Gamma\subset \Psldc$ 
as a group of isometries of $\hh^3$, and $h_0$ is the hyperbolic metric on
 $\Gamma\backslash\Omega_\Gamma\simeq \pl \oX$. For points in
$\mathfrak{S}$ corresponding to Schottky groups $\Gamma$ with
dimension of limit set $\delta_\Gamma<1$, the function $F(\Gamma)$ is given by the
following absolutely convergent product:
\begin{equation}\label{e:Zog-F}
F(\Gamma)= \prod_{\{\gamma\}}\prod^\infty_{m=0} (1-
q_\gamma^{1+m})
\end{equation}
where $q_\gamma$ is the multiplier of $\gamma\in\Gamma$, and
$\{\gamma\}$ runs over all distinct primitive conjugacy classes in
$\Gamma$ excluding the identity.
\end{theorem}

\begin{remark} The formula \eqref{e:Zograf} was in fact given in terms of Liouville action $S$ instead of renormalized volume, but  it has been shown that $S=-4\,{\rm Vol}_R(X)+c_\bg$ for some constant $c_\bg$ depending only on $\bg$, by Krasnov \cite{Kra} for Schottky manifolds and by Takhtajan-Teo  \cite{TakTeo} for quasi-Fuchsian manifolds.
\end{remark}

We  therefore deduce from this last theorem and our construction the following
\begin{theorem}
On the Schottky space $\mathfrak{S}$,  the bundle $\mc{L}_{\mathfrak{S}}^{-1}$ is isomorphic
to $(\det \pl)^{\otimes 6}$ when equipped with their connections and Hermitian products induced by those of $(\mc{L}_{\mathfrak{S}},\nabla^{\mathfrak{S}}, \|\cdot\|_{\mc{L}})$  and 
$(\det \pl, \nabla^{\det},\|\cdot\|_{Q})$.
There is an explicit isometric isomorphism of holomorphic Hermitian line bundles given by   
\[ 
 (\sqrt{c_\bg} F \varphi)^{\otimes 6}  \mapsto e^{-2\pi i\CSP}. 
\]  
where $F$ and $c_\bg$ are  the holomorphic functions and constants of Theorem \ref{thzograf}, $\varphi$ is the canonical section of $\det \pl$ defined in \eqref{cansection}. 
\end{theorem} 
\begin{proof} The section $ e^{2\pi i\CSP}\otimes\sqrt{c_\bg} F \varphi)^{\otimes 6} $  
is holomorphic and has constant norm in the Hermitian line bundle $\mc{L}_{\mathfrak{S}} \otimes (\det \pl)^{\otimes 6}$ which is flat with respect to the Hermitian connection, then it is parallel and provides an 
isomorphism with the trivial line bundle. 
\end{proof}

\begin{remark}
Notice that the function $F$ defined by the product \eqref{e:Zog-F} when  $\delta(\Gamma)<1$ is known to extend  analytically in $\mathfrak{S}$ by results of Zograf \cite{Zo1,Zo2}, and our theorem provides another proof, assuming formula \eqref{e:Zograf} only in the subset $\{\Gamma\in \mathfrak{S}; \delta(\Gamma)<1\}$.   
\end{remark}
\begin{remark} 
In our previous work \cite{GMP}, we proved that 
\[F(\Gamma)=|F(\Gamma)| \exp\left(-\tfrac{\pi i}{2}\eta(A)\right)\] 
when $\delta(\Gamma)<1$, where $\eta(A)$ is the eta invariant of the signature operator $A=*d+d*$ on odd dimensional forms on the Schottky $3$-manifold $\Gamma\backslash\hh^3$.
\end{remark} 
\begin{remark}
Using the result of McIntyre-Takhajan and McIntyre-Teo \cite{McTa,McTe}, a similar result with different powers of the bundles is easily obtained in the Schottky and quasi-Fuchsian cases if one replaces the bundle $\det \pl$ by the determinant line bundle $\det \Lambda_n$ of the vector space of holomorphic $n$-differentials on $M$.
\end{remark}

\appendix
\section{Chern-Simons invariants of $3$-manifolds with funnels and cusps of rank $2$}
In this appendix we show how to extend the results of this paper to include $3$-manifolds of finite geometry with funnels as well as rank $2$-cusps. We will concentrate on the cusps since funnels have already been treated.

By definition, a \emph{cusp of maximal rank} is a half-complete warped product $(a,\infty)\times M$ with metric $dt^2+e^{-2t} h$, where $h$ is a flat metric on $M$. Here $M$ will be of dimension $2$. After a linear change of variables in $t$, we can thus assume that $M$ is isometric to a flat torus with a closed simple geodesic of length $1$. 

By changing variables $x:=e^{-t}\in (0,e^{-a})$, the cusp metric becomes
\[\frac{dx^2}{x^2}+x^2 h=x^2\left(\frac{dx^2}{x^4}+h\right).\]
Thus a cusp is conformal to a half-infinite cylinder $dy^2+h$ where $y:=x^{-1}=e^{t}\in [e^a,\infty)$, the conformal factor being $x=y^{-1}$. The function $x$ can be used to glue to the cusps a copy of $M$ at $x=0$, thus compactifying $X$. Thus if we choose $\rho: X\to (0,\infty)$ to be a function which agrees with $x$ on funnels and with $y$ on cusps, it follows that $X$ is conformal to a manifold with boundary (corresponding to the funnels) and flat half-infinite cylindrical ends (corresponding to each cusp):
\begin{align*}
g=\rho^{-2} \hat{g}, &&\hat{g}=d\rho^2+h(\rho)
\end{align*}
where on the cusps, $h(\rho)=h$ is flat and independent of $\rho$. 

Let $\hS$ be a orthonormal frame for $\hat{g}$ which is parallel in the $y$ direction in the cusp. Then both the connection $1$-form $\homega$ and the curvature form $\hat{\Omega}$ vanish when contracted with $\partial_y$.
It follows that the Chern-Simons form $\cs(\hat{g},\hS)$ vanishes identically on the cusp, thus the 
$\Sot$ Chern-Simons invariant for $\hg$ is well-defined and moreover it coincides with the invariant of the compact manifold with boundary obtained by chopping off the cylindrical ends.

The line bundle $\cL$ is constructed now over the set of constant-curvature metrics on $M$, namely hyperbolic on the funnel ends and flat on the cusp ends. In the definition of the cocycle $c^X(\hS,a)$ notice that the second term vanishes identically on the cusp, since we work with frames $S$ parallel in the direction of $y$, which implies that
$\partial_y \tilde{a}=0$, or in other words $\tilde{a}$ is independent of $y$. 
The definition of the $\Sot$ connection is unchanged if we include now in $M$ the flat components corresponding to the cusps. Its curvature is computed in terms of a fiberwise integral of the Pontrjagin form by following verbatim the proof of Proposition \ref{curvcslb}. However in Lemma \ref{idecu} the curvature of the tori fibers vanishes, thus the cusps do not contribute to the curvature and so the curvature of $\nabla^{\mc{L}}$ is $\frac{i}{8\pi }$ times the Weil-Petersson symplectic form of the Teichm\"uller space corresponding to the funnels, i.e., it does not ``see'' the cusps.

We define now the $\Sot$ invariant of the hyperbolic metric $g$.
Using \eqref{csformcf} with the roles of $g,\hg$ reversed and \eqref{omegavs'} we see that in the cusp, the Chern-Simons form $\cs(g,S)$ of $g$ equals $d\Tr(\hat{\alpha}\wedge\homega)$, where
\[\hat{\alpha}_{ij}(Y)=y^{-1}[\hg(Y,\hS_j)S_i(y)-\hg(Y,S_i)S_j(y)].\]
Now $\homega_{ij}$ is constant in $y$ in the sense that $\cL_{\partial_y} \homega_{ij}=0$, while
$\hat{\alpha}$ is of homogeneity $-1$. It follows that $\cs(g,S)$
decreases like $y^{-2}$ as $y\to\infty$, thus it is integrable without regularization. Moreover the form $\Tr(\hat{\alpha}\wedge\homega)$ from \eqref{csformcf} is homogeneous in $y$ of degree $-1$, hence Proposition \ref{CSconforme} continues to hold in the setting of this appendix.

To define the $\Psldc$ invariant we use Proposition \ref{chsomega^c}. We note that the volume of the cusps is finite, the $\Sot$ Chern-Simons form was proved above to be integrable in the cusp, and we claim that the remaining term $\Tr(T\wedge\omega)$ decreases in the cusp like $y^{-1}$. Indeed, we have seen above that $\omega=\homega+\hat{\alpha}$ is of homogeneity $0$ and $-1$, while $T=y^{-1}\hat{T}$ is of homogeneity $-1$.
Therefore $\CS^\Psldc$ does not involve regularization in the cusps, 
while Proposition \ref{comparaisons} continues to hold. Note that the Euler characteristic of a torus is $0$, so it is irrelevant whether the tori closing the cusps are included or not in the formula from Proposition \ref{comparaisons} when we allow cusps.

The variation formula for $\CS^\Psldc$ (Theorem \ref{variationPSL2}) continues to hold as in the case without cusps.
because in \eqref{varcspsld} the cusp
terms involved (other than the first one which is the connection $1$-form) do not have contributions of degree $0$ in $y$. This is obvious if one takes into account that $\alpha$ and $T$ are of homogeneity $-1$, while $\homega$ is of homogeneity $0$.
Hence the variation of the regularized volume of a hyperbolic manifold with funnels and cusps is given by Corollary \ref{0vol} (and only depends on local data on the funnels).

Finally, the correspondence between hyperbolic metrics on $X$ and the conformal infinity in the funnels continues to hold in the presence of cusps \cite{Marden}. 

These hyperbolic metrics with cusps and funnels form therefore a Lagrangian submanifold in $T\Teig$, and their renormalized volume is a K\"ahler potential for the Teichm\"uller space corresponding to the funnels (see Corollary \ref{kahlerpot}). 

\subsection*{Acknowledgements}
The subject of this paper arose from our joint work \cite{GMP} with J.\ Park, to whom we owe the idea of connecting the determinant and Chern-Simons line bundles in this context. 
We thank him and also S.~Baseilhac, U.~Bunke, C.~Ciobotaru, K.~Krasnov, J. March\'e, G.~Massuyeau, and J.-M.~Schlenker for useful discussions.
C.G.\ is supported by grant ANR-09-JCJC-0099-01, S.M.\ is partially supported by grant PN-II-ID-PCE 1188 265/2009. This work was done while 
S.M.\ was a visiting researcher for CNRS at the DMA of the Ecole Normale Sup\'erieure and a visitor at IHES, he thanks these institutions for their support.

\end{document}